\documentclass[onefignum,onetabnum]{siamart190516}

\usepackage{}
\usepackage{graphicx}
\usepackage{multirow,bigstrut}
\usepackage{amsmath,amsfonts,amssymb}
\usepackage{mathrsfs}
\usepackage{color}
\usepackage{xcolor}
\usepackage{upgreek}
\usepackage{bm}
\usepackage{verbatim}
\usepackage{mathtools}
\usepackage{calligra}
\usepackage[T1]{fontenc}
\usepackage{subfigure}
\usepackage{pgfplots}
\usepackage{amsmath,accents}
\newcommand{\cm}[1]{{\color{black}{#1}}}

\usepackage{tikz}
  \usetikzlibrary{calc,shapes,decorations,decorations.text, mindmap,shadings,patterns,matrix,arrows,intersections,automata,backgrounds}
  
\usetikzlibrary{shapes,arrows}

\usepackage{verbatim}
\usepackage{exscale}
\usepackage{relsize}
\newtheorem{rem}[theorem]{\hspace{1mm}Remark}

\newtheorem{assumption}[theorem]{Assumption}

\makeatletter
\tikzset{
        hatch distance/.store in=\hatchdistance,
        hatch distance=5pt,
        hatch thickness/.store in=\hatchthickness,
        hatch thickness=5pt
        }
\pgfdeclarepatternformonly[\hatchdistance,\hatchthickness]{north east hatch}
    {\pgfqpoint{-1pt}{-1pt}}
    {\pgfqpoint{\hatchdistance}{\hatchdistance}}
    {\pgfpoint{\hatchdistance-1pt}{\hatchdistance-1pt}}%
    {
        \pgfsetcolor{\tikz@pattern@color}
        \pgfsetlinewidth{\hatchthickness}
        \pgfpathmoveto{\pgfqpoint{0pt}{0pt}}
        \pgfpathlineto{\pgfqpoint{\hatchdistance}{\hatchdistance}}
        \pgfusepath{stroke}
    }
\makeatother

\newcommand{\interior}[1]{\accentset{\smash{\raisebox{-0.12ex}{$\scriptstyle\circ$}}}{#1}\rule{0pt}{2.3ex}}

\begin{document}
 

\headers{Error estimates for discrete GFEMs}{C. P. Ma and R. Scheichl} 
 
\title{Error estimates for fully discrete generalized FEMs with locally optimal spectral approximations}

\author{Chupeng Ma\thanks 
    {Institute for Applied Mathematics and Interdisciplinary Center for Scientific Computing, Heidelberg University, Im Neuenheimer Feld 205, Heidelberg 69120, Germany (\email{chupeng.ma@uni-heidelberg.de}, \email{r.scheichl@uni-heidelberg.de}).}
 \and R. Scheichl\footnotemark[1]
}
\maketitle

\begin{abstract}
This paper is concerned with error estimates of the fully discrete generalized finite element method (GFEM) with optimal local approximation spaces for solving elliptic problems with heterogeneous coefficients. The local approximation spaces are constructed using eigenvectors of local eigenvalue problems solved by the finite element method on some sufficiently fine mesh with mesh size $h$. The error bound of the discrete GFEM approximation is proved to converge as $h\rightarrow 0$ towards that of the continuous GFEM approximation, which was shown to decay nearly exponentially in previous works. Moreover, even for fixed mesh size $h$, a nearly exponential rate of convergence of the local approximation errors with respect to the dimension of the local spaces is established. An efficient and accurate method for solving the discrete eigenvalue problems is proposed by incorporating the discrete $A$-harmonic constraint directly into the eigensolver. Numerical experiments are carried out to confirm the theoretical results and to demonstrate the effectiveness of the method.
\end{abstract}



\section{Introduction}\label{sec-1}
Multiscale problems are ubiquitous in science and engineering. Typical examples are flow and transport phenomena in heterogeneous porous media which usually span over a wide range of length scales, from the pore scale ($\sim$\,$\mathrm{\mu}$m), to the Darcy scale ($\sim$\,cm), all the way to the field scale ($\sim$\,km). A direct numerical simulation of multiscale problems with standard methods, taking into account all fine-scale details, is prohibitively expensive, because a very fine spatial discretization would be required. In order to resolve large-scale features of solutions without "pollution" due to the corresponding fine-scale details, various multiscale methods have been developed in different communities in the past several decades. Here we focus on multiscale methods that are based on numerical upscaling by incorporating fine-scale information into coarse-scale models. Such methods include the multiscale finite element method (MsFEM) \cite{hou1997multiscale,chen2003mixed,efendiev2009multiscale}, the generalized multiscale finite element method (GMsFEM) \cite{efendiev2013generalized,efendiev2014generalized,chung2015mixed}, the heterogeneous multiscale method (HMM) \cite{ming2005analysis,weinan2007heterogeneous}, and the localized orthogonal decomposition (LOD) \cite{maalqvist2014localization,henning2014localized}, to cite a few. 

In this paper, we concern ourselves with another related multiscale method, the Multiscale Spectral Generalized Finite Element Method (MS-GFEM) proposed in the pioneering paper of Babuska and Lipton \cite{babuska2011optimal}. It is a generalized finite element method (GFEM) \cite{babuvska1997partition,melenk1995generalized} with optimal local approximation spaces constructed by the eigenvectors of local eigenvalue problems. These local approximation spaces are glued together by a partition of unity to form the trial space on which the coarse problem is solved. In \cite{babuska2011optimal}, an eigenvalue problem associated with the singular value decomposition of a restriction operator was used to build the local approximation spaces for solving second-order elliptic problems with heterogeneous coefficients. A nearly exponential decay rate of the local approximation errors with respect to the dimension of the local spaces was established. The method was later generalized to solve elasticity problems \cite{babuvska2014machine} and the numerical implementation of the method was discussed in \cite{babuvska2020multiscale}. In our recent work \cite{ma2021novel}, the results in \cite{babuska2011optimal} were extended in several respects:\;($i$) a new local eigenvalue problem, directly incorporating the partition of unity function was used to construct the optimal local approximation spaces with some advantages over those in \cite{babuska2011optimal}; ($ii$) a sharper error bound of the local approximations was derived that demonstrates the influence of the oversampling size; ($iii$) the theoretical results in \cite{babuska2011optimal} were extended to problems with mixed boundary conditions defined on general Lipschitz domains.

In practical computations, the local eigenvalue problems in the MS-GFEM need to be solved by numerical methods (e.g., FEM), giving rise to the discrete MS-GFEM. However, the theoretical results in \cite{babuska2011optimal,ma2021novel} are only concerned with the error estimates of the method at the continuous level. In \cite{babuvska2014machine}, the error estimates of the discrete MS-GFEM for elasticity problems with a FE approximation of the local eigenvalue problems were discussed. Due to a heavy reliance on the approximation results for the continuous eigenfunctions in the analysis, the error estimates in \cite{babuvska2014machine} hold only under some unjustified assumptions. A rigorous theoretical analysis of the discrete MS-GFEM without artificial assumptions is thus still missing. 

Another important issue concerning the discrete MS-GFEM is the high cost associated with the construction of the discrete $A$-harmonic spaces on which the local eigenvalue problems are posed. These spaces are spanned by the discrete $A$-harmonic extensions of the hat functions associated with the nodes on the subdomain boundary and thus in general, a large number of local boundary value problems need to be solved. In \cite{babuska2011optimal,babuvska2020multiscale}, it was suggested to reduce this cost by restricting to the span of the discrete $A$-harmonic extensions of some special functions defined on the boundary, thus only approximating the discrete $A$-harmonic spaces. Random sampling techniques were also exploited to generate the basis functions of these spaces \cite{calo2016randomized,chen2020randomized}. In \cite{ma2021novel}, approximations of the discrete $A$-harmonic spaces were constructed using eigenfunctions of a Steklov eigenvalue problem. All these strategies are based on approximating the discrete $A$-harmonic spaces and thus have two drawbacks:\;(i) additional errors in approximating the discrete $A$-harmonic spaces are introduced to the method and (ii) the inefficiencies are not really resolved. Inefficiency also arises in the context of a practical interesting adaptive implementation of the method. In particular, when more eigenfunctions are required to enrich the local approximation spaces, the approximations of the discrete $A$-harmonic spaces also may need to be enriched or reconstructed before the eigenvalue problems can be solved again.

In this paper, the error estimates of the discrete MS-GFEM with approximation of the local eigenvalue problems in some fine-scale finite element space $V_{h}$ ($h$ denoting the mesh size) are derived without any unjustified assumptions. It is shown that the error of the discrete MS-GFEM is bounded by the sum of the FE error in $V_{h}$ and the errors arising from the local approximations of this fine-scale FE solution. We prove that as $h\rightarrow 0$, the local errors in the discrete MS-GFEM converge towards those in the continuous MS-GFEM, which were shown to decay nearly exponentially in \cite{ma2021novel}. As a result, the global error bound of the discrete MS-GFEM converges towards that of the continuous method as $h\rightarrow 0$. While the analysis in \cite{babuvska2014machine} relies heavily on bounding the eigenfunction error, our analysis is based on the convergence of the eigenvalues of the discrete eigenvalue problems, using a theoretical framework developed in the context of homogenization theory \cite{jikov2012homogenization}. There are two interesting features of the local eigenvalue problems that make the convergence analysis challenging. Firstly, due to the $A$-harmonicity, the discrete eigenvalue problems are non-conforming approximations of the continuous eigenvalue  problems. Secondly, the compact operators associated with the discrete eigenvalue problems are defined on function spaces that differ significantly from those used in classical FE approximations of PDE eigenvalue problems. Finally, similar to the continuous problem, we establish a nearly exponential decay rate of the local errors with respect to the dimension of the local spaces in the fully discrete setting, which is missing in previous studies. Due to the nearly exponential decay rate, the discrete MS-GFEM approximation can yield good results with a very small number of eigenfunctions per subdomain.

In addition to providing new error estimates for the discrete MS-GFEM, an efficient and accurate method for solving the discrete eigenvalue problems is developed in this paper. In our method, the eigenvalue problem is first rewritten in mixed form by introducing a Lagrange multiplier to relax the $A$-harmonic constraint. Instead of solving the discrete eigenvalue problem in mixed formulation directly, we perform a Cholesky (LDL) factorization and then solve an equivalent reduced problem with half the size. By incorporating the $A$-harmonic constraint into the eigenproblems instead of approximating the $A$-harmonic spaces in advance, our method is more accurate and efficient than all previous techniques. Furthermore, it paves the way for an efficient implementation of truly adaptive versions of the MS-GFEM. Indeed, provided the solutions of the discrete eigenvalue problems are sufficiently accurate, the local approximation error in each subdomain is controlled by the eigenvalue corresponding to the first eigenfunction that is not included in the local approximation space. This enables an adaptive selection of the number of eigenfunctions used in each subdomain since an upper bound of the local approximation error can be estimated reliably $a \;posteriori$ from the computed eigenvalues.

The rest of this paper is organized as follows. In \cref{sec-2}, the problem considered in this paper and the (continuous) MS-GFEM are introduced. In \cref{sec-3}, the discrete MS-GFEM is described in detail and some technical tools used in the error analysis are provided. \Cref{sec-4} is devoted to the convergence analysis of the eigenvalue problems and the proof of the nearly exponential decay rate of the local approximation errors. The new technique for solving the discrete eigenvalue problems is presented in \cref{sec-5}. Numerical experiments are provided in \cref{sec-6} to verify the theoretical analysis and to demonstrate the effectiveness of the method.

\section{Problem formulation and the MS-GFEM}\label{sec-2}
\subsection{Problem formulation}
Let $\Omega\subset \mathbb{R}^{d}$, $d=2,3$ be a bounded polygonal ($d=2$) or polyhedral ($d=3$) domain with Lipschitz boundary $\partial \Omega$. We consider the elliptic partial differential equation with mixed boundary conditions
\begin{equation}\label{eq:1-1}
\left\{
\begin{array}{lll}
{\displaystyle -{\rm div}(A({\bm x})\nabla u({\bm x})) = f({\bm x}),\quad {\rm in}\;\, \Omega }\\[2mm]
{\displaystyle {\bm n} \cdot A({\bm x})\nabla u({\bm x})=g({\bm x}), \quad \quad \;\;{\rm on}\;\,\partial \Omega_{N}}\\[2mm]
{\displaystyle u({\bm x}) = q({\bm x}), \quad \qquad \qquad \qquad {\rm on}\;\,\partial \Omega_{D},}
\end{array}
\right.
\end{equation}
where ${\bm n}$ denotes the unit outward normal vector to the boundary, $\partial\Omega_{D}\cap \partial\Omega_{N} = \emptyset$, and $\overline{\partial\Omega_{D}}\cup \overline{\partial\Omega_{N}} = \partial \Omega$. Throughout the paper we make the following assumptions on the problem:

\begin{assumption}\label{ass:1} $(i)$ $A({\bm x}) \in (L^{\infty}(\Omega))^{d\times d}$ is pointwise symmetric and there exists $0< \alpha < \beta$ such that
\begin{equation}\label{eq:1-1-0}
\alpha |{\bm \xi}|^{2} \leq A({\bm x}){\bm \xi}\cdot{\bm \xi} \leq \beta  |{\bm \xi}|^{2},\quad \forall {\bm \xi}\in \mathbb{R}^{d},\quad {\bm x} \in\Omega;
\end{equation}

$(ii)$ $f\in L^{2}(\Omega)$, $g\in H^{-1/2}(\partial \Omega_{N})$, $q\in H^{1/2}(\partial \Omega_{D})$.
\end{assumption}

The weak formulation of problem \cref{eq:1-1} is to find $u^{e}\in H^{1}(\Omega)$ with $u^{e} = q({\bm x})$ on $\partial \Omega_{D}$ such that 
\begin{equation}\label{eq:1-2}
a(u^{e},v) = F(v),\quad \forall v\in H^{1}_{D}(\Omega),
\end{equation}
where 
\begin{equation}\label{eq:1-2-0}
H^{1}_{D}(\Omega) = \big\{v\in H^{1}(\Omega)\;:\; v = 0 \;\;{\rm on}\;\,\partial \Omega_{D}\big\},
\end{equation}
and the bilinear form $a(\cdot,\cdot)$ and the functional $F$ are defined by
\begin{equation}\label{eq:1-2-1}
a(u,v) = \int_{\Omega} A({\bm x})\nabla u\cdot {\nabla v} \,d{\bm x},\quad F(v) = \int_{\partial \Omega_{N}} gv \,d{\bm s}+ \int_{\Omega}fv \,d{\bm x}.
\end{equation}
Under \cref{ass:1}, the weak formulation of problem \cref{eq:1-1} is uniquely solvable. For later use, we define 
\begin{equation}\label{eq:1-2-2}
a_{\omega}(u,v) = \int_{\omega} A({\bm x})\nabla u\cdot {\nabla v} \,d{\bm x}, \quad F_{\omega}(v) = \int_{\partial \omega\cap\partial \Omega_{N}} gv \,d{\bm s}+ \int_{\omega}fv \,d{\bm x},
\end{equation}
and $\Vert u \Vert_{a,\,\omega} = \sqrt{a_{\omega}(u,u)}$ for any subdomain $\omega\subset \Omega$ and $u$, $v\in H^{1}(\omega)$. If $\omega = \Omega$, we drop the domain and write $a(\cdot,\cdot)$ and $\Vert\cdot\Vert_{a}$ instead of $a_{\Omega}(\cdot,\cdot)$ and $\Vert\cdot\Vert_{a,\,\Omega}$.

We now briefly describe the GFEM for solving \cref{eq:1-2}. Let $\{ \omega_{i} \}_{i=1}^{M}$ be a collection of open sets satisfying $\omega_{i}\subset\Omega$ and $\cup_{i=1}^{M} \omega_{i} = \Omega$. We assume that each point ${\bm x}\in\Omega$ belongs to at most $\kappa$ subdomains $\omega_{i}$. Let $\{ \chi_{i} \}_{i=1}^{M}$ be a partition of unity subordinate to the open covering satisfying the following properties:
\begin{equation}\label{eq:1-3}
\begin{array}{lll}
{\displaystyle 0\leq \chi_{i}({\bm x})\leq 1,\quad \sum_{i=1}^{M}\chi_{i}({\bm x}) =1, \quad \forall \,{\bm x}\in \Omega,}\\[4mm]
{\displaystyle \chi_{i}({\bm x})= 0, \quad \forall \,{\bm x}\in \Omega/\omega_{i}, \quad i=1,\cdots,M,}\\[2mm]
{\displaystyle \chi_{i}\in C^{1}(\omega_{i}),\;\;\max_{{\bm x}\in\Omega} |\nabla \chi_{i}({\bm x}) | \leq \frac{C_{1}}{diam\,(\omega_{i})},\quad i=1,\cdots,M.}
\end{array}
\end{equation}

For $i=1,\cdots,M$, let $u^{p}_{i} \in H^{1}(\omega_{i})$ be a local particular function and $S_{n_{i}}(\omega_{i})\subset H^{1}(\omega_{i})$ be a local approximation space of dimension $n_{i}$. For a subdomain $\omega_{i}$ that intersects the Dirichlet boundary $\partial \Omega_{D}$, we additionally assume that $u^{p}_{i} = q({\bm x})$ on $\partial \omega_{i}\cap \partial \Omega_{D}$ and that functions in $S_{n_{i}}(\omega_{i})$ vanish on $\partial \omega_{i}\cap \partial \Omega_{D}$. A key step of the GFEM is to build the global particular function $u^{p}$ and the trial space $S_{n}(\Omega)$ by pasting together the local particular functions and the local approximation spaces by using the partition of unity:
\begin{equation}\label{eq:1-3-0}
\begin{array}{lll}
{\displaystyle u^{p} = \sum_{i=1}^{M}\chi_{i}u^{p}_{i},\quad  S_{n}(\Omega) =\Big\{\sum_{i=1}^{M}\chi_{i}\phi_{i}\,:\, \phi_{i}\in S_{n_{i}}(\omega_{i})\Big\}. }
\end{array}
\end{equation}
It is easy to see that $u^{p}\in H^{1}(\Omega)$, $u^{p}=q({\bm x})$ on $\partial \Omega_{D}$, and $S_{n}(\Omega)\subset H^{1}_{D}(\Omega)$. The following theorem allows us to derive the global approximation error from the local approximation errors.
\begin{theorem}[\cite{ma2021novel}]\label{thm:1-0}
Assume that there exists $\xi_{i}\in S_{n_{i}}(\omega_{i})$, $i=1,\cdots,M$, such that
\begin{equation}\label{eq:1-3-1}
\big\Vert \chi_{i}(u^{e}-u^{p}_{i}-\xi_{i})\big\Vert_{a,\,\omega_{i}}\leq \varepsilon_{i}\Vert u^{e}\Vert_{a,\,\omega^{\ast}_{i}},
\end{equation}
where $\omega_{i}\subset\omega_{i}^{\ast}\subset \Omega$. Let $\Psi = u^{p} + \sum_{i=1}^{M}\chi_{i}\xi_{i}$. Then
\begin{equation}\label{eq:1-3-2}
\begin{array}{lll}
{\displaystyle \big\Vert u^{e} - \Psi \big \Vert_{a} \leq \sqrt{\kappa\kappa^{\ast}}\big(\max_{i=1,\cdots,M}\varepsilon_{i}\big)\Vert u^{e}\Vert_{a}.}
\end{array}
\end{equation}
Here we assume that each point ${\bm x}\in\Omega$ belongs to at most $\kappa^{\ast}$ subdomains $\omega_{i}^{\ast}$. 
\end{theorem}

The final step of the GFEM is to solve the problem \cref{eq:1-2} on the finite-dimensional approximation space $S_{n}(\Omega)$. We seek the approximate solution $u^{G}= u^{p} + u^{s}$ with $u^{s}\in S_{n}(\Omega)$ such that
\begin{equation}\label{eq:1-2-3}
a(u^{s}, v) = F(v) - a(u^{p}, v),\quad \forall v\in S_{n}(\Omega).
\end{equation}
It is a classical result that 
\begin{equation}\label{eq:1-2-4}
u^{G} = {\rm argmin}\{\Vert u^{e} - v\Vert_{a}\,:\, v\in u^{p} + S_{n}(\Omega)\},
\end{equation}
where $u^{e}$ is the solution of \cref{eq:1-2}. 
It follows from \cref{thm:1-0} and \cref{eq:1-2-4} that the error of the GFEM for solving \cref{eq:1-2} is bounded by 
\begin{equation}\label{eq:1-3-4}
\displaystyle \big\Vert u^{e} - u^{G} \big\Vert_{a} \leq  \big\Vert u^{e} - \Psi\big\Vert_{a}\leq \sqrt{\kappa \kappa^{\ast}}\big(\max_{i=1,\cdots,M}\varepsilon_{i}\big)\Vert u^{e}\Vert_{a}.
\end{equation}
\Cref{eq:1-3-4} indicates that the error of the GFEM is determined by the local approximation errors. A key feature of the MS-GFEM is the construction of the local approximation spaces by local eigenvalue problems such that the local approximation errors $\varepsilon_{i}$ decay nearly exponentially with respect to $n_{i}$. This is described in what follows.

\subsection{Local particular functions and optimal local approximation spaces}
In this subsection, we review the local particular functions and the optimal local approximation spaces proposed in \cite{ma2021novel} for the MS-GFEM, for which the local errors decay nearly exponentially.

\begin{figure}\label{fig:2-1}
\centering
\includegraphics[scale=0.5]{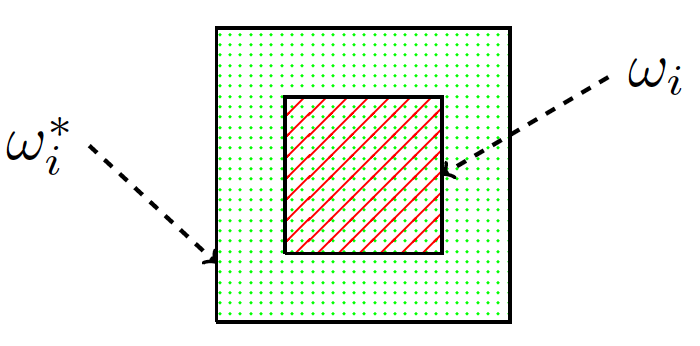}~
\includegraphics[scale=0.5]{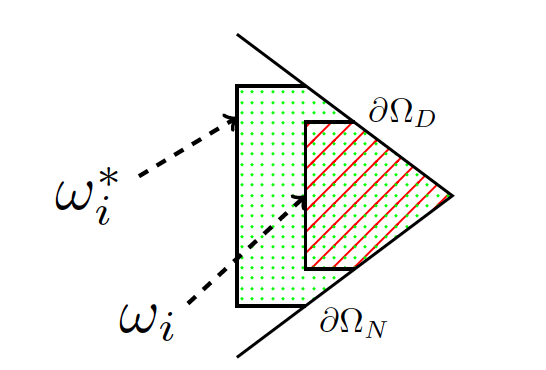}
\caption{Illustration of a subdomain $\omega_i$ that lies within the interior of $\Omega$ (left) and one that intersects the boundary of $\Omega$ (right) with associated oversampling domains $\omega_i^{\ast}$.}
\end{figure}

For a subdomain $\omega_{i}$, we introduce another domain $\omega_{i}^{\ast}$ that satisfies $\omega_{i} \subset \omega_{i}^{\ast}\subset\Omega$ as illustrated in \cref{fig:2-1}. $\omega_{i}^{\ast}$ is usually referred to as the \textit{oversampling domain}. On each $\omega_{i}^{\ast}$, we define the following subspaces of $H^{1}(\omega_{i}^{\ast})$:
\begin{equation}\label{eq:2-2}
\begin{array}{lll}
{\displaystyle H^{1}_{D}(\omega^{\ast}_{i}) = \big\{v\in H^{1}(\omega^{\ast}_{i})\;:\;v = 0\;\;{\rm on}\;\,  \partial\omega^{\ast}_{i} \cap \partial \Omega_{D}\big\}, }\\[2mm]
{\displaystyle H^{1}_{DI}(\omega^{\ast}_{i}) = \big\{v\in H^{1}(\omega^{\ast}_{i})\;:\;v = 0\;\;{\rm on}\;\, (\partial\omega^{\ast}_{i} \cap \partial \Omega_{D})\cup (\partial \omega^{\ast}_{i}\cap \Omega)\big\},}
\end{array}
\end{equation}
and the $A$-harmonic space
\begin{equation}\label{eq:2-1}
H_{A,D}(\omega^{\ast}_{i}) = \big\{ u\in H^{1}_{D}(\omega^{\ast}_{i})\;:\;  a_{\omega^{\ast}_{i}}(u, v) = 0 \quad \forall \,v\in H^{1}_{DI}(\omega^{\ast}_{i})\big\}.
\end{equation}

To take care of the right hand side and the boundary conditions, we introduce a function $\psi_{i} = \psi_{i}^{r} + \psi_{i}^{d}$, where $\psi_{i}^{r}$ and $\psi_{i}^{d}$ are the weak solutions of 
\begin{equation}\label{eq:2-3}
\left\{
\begin{array}{lll}
{\displaystyle -{\rm div} (A({\bm x})\nabla \psi^{r}_{i}({\bm x})) = f({\bm x}), \quad {\rm in}\;\, \omega^{\ast}_{i}  }\\[2mm]
{\displaystyle {\bm n}\cdot A({\bm x})\nabla\psi^{r}_{i}({\bm x}) = g({\bm x}),\qquad \;\; {\rm on}\;\, \partial \omega^{\ast}_{i} \cap \partial \Omega_{N}}\\[2mm]
{\displaystyle \psi^{r}_{i}({\bm x}) = 0,\quad \quad \quad \qquad \qquad \quad \;\;{\rm on}\;\, \partial \omega^{\ast}_{i}\cap \Omega}\\[2mm]
{\displaystyle \psi^{r}_{i}({\bm x}) = 0,\;\quad \quad \qquad \qquad \qquad \;{\rm on}\;\, \partial \omega^{\ast}_{i} \cap \partial \Omega_{D}}
\end{array}
\right.
\end{equation}
and 
\begin{equation}\label{eq:2-4}
\left\{
\begin{array}{lll}
{\displaystyle -{\rm div} (A({\bm x})\nabla \psi^{d}_{i}({\bm x})) = 0, \quad \,{\rm in} \;\,\omega^{\ast}_{i}  }\\[2mm]
{\displaystyle {\bm n}\cdot A({\bm x})\nabla\psi^{d}_{i}({\bm x}) = 0,\qquad \;\;\, {\rm on}\;\, \partial \omega^{\ast}_{i} \cap \partial \Omega_{N}}\\[2mm]
{\displaystyle {\bm n}\cdot A({\bm x})\nabla\psi^{d}_{i}({\bm x}) = 0,\quad \quad \;\; \;{\rm on} \;\,\partial \omega^{\ast}_{i}\cap \Omega}\\[2mm]
{\displaystyle \psi^{d}_{i}({\bm x}) = q({\bm x}),\;\;\;\quad \qquad \qquad {\rm on} \;\,\partial \omega^{\ast}_{i} \cap \partial \Omega_{D},}
\end{array}
\right.
\end{equation}
respectively. Note that if $\partial \omega^{\ast}_{i} \cap \partial \Omega_{D} = \emptyset$, then $\psi^{d}_{i}$ vanishes. It was proved in \cite{ma2021novel} that
\begin{equation}\label{eq:2-4-0}
u^{e}|_{\omega_{i}^{\ast}}-\psi_{i}\in H_{A,D}(\omega^{\ast}_{i}),\quad a_{\omega_{i}^{\ast}}(u^{e}|_{\omega_{i}^{\ast}}-\psi_{i}, \,\psi_{i}) = 0,
\end{equation}
where $u^{e}$ is the solution of \cref{eq:1-2}. We see that the solution is locally decomposed into two orthogonal parts with respect to the energy inner product $a_{\omega_{i}^{\ast}}(\cdot,\cdot)$. The first part is the sum of the solutions of two local boundary value problems (one vanishes if $\partial \omega^{\ast}_{i} \cap \partial \Omega_{D} = \emptyset$). The other part lies in the $A$-harmonic space. In \cite{ma2021novel}, an optimal local approximation space was constructed from singular vectors of a compact operator involving the partition of unity function $\chi_i$ that allows to approximate the $A$-harmonic part with a nearly exponential decay rate.

For each $k\in\mathbb{N}$, let $\lambda_{k}$ and $\phi_{k}$ denote the $k$-th eigenvalue and the associated eigenfunction (arranged in increasing order) of the problem
\begin{equation}\label{eq:2-5}
a_{\omega_{i}^{\ast}}(\phi, v) = \lambda\,a_{\omega_{i}}(\chi_{i} \phi, \chi_{i}v),\quad \forall v\in H_{A,D}(\omega_{i}^{\ast}),
\end{equation}
where $\chi_{i}$ is the partition of unity function supported on $\omega_{i}$. 
Next we define the local particular functions and the local approximation spaces for the MS-GFEM.

\begin{theorem}[\cite{ma2021novel}]\label{thm:2-1}
For $i=1,\cdots,M$, let the local particular function and the local approximation space on $\omega_{i}$ be defined as
\begin{equation}
u^{p}_{i}:=\psi_{i}|_{\omega_{i}},\quad S_{n_{i}}(\omega_{i}) := {\rm span}\{\phi_{1}|_{\omega_{i}},\cdots, \phi_{n_{i}}|_{\omega_{i}}\},
\end{equation}
where $\psi_{i} = \psi_{i}^{r} + \psi_{i}^{d}$ with $\psi_{i}^{r}$ and $\psi_{i}^{d}$ defined in \cref{eq:2-3,eq:2-4}, and let $u^{e}$ be the solution of \cref{eq:1-2}. Then, there exists a $\xi_{i}\in S_{n_{i}}(\omega_{i})$ such that
\begin{equation}
\Vert \chi_{i}(u^{e} - u^{p}_{i} - \xi_{i})\Vert_{a,\omega_{i}}\leq \lambda^{-1/2}_{n_{i}+1}\,\Vert u^{e}\Vert_{a,\omega_{i}^{\ast}}.
\end{equation}
\end{theorem}

\Cref{thm:2-1} shows that the local approximation errors in the MS-GFEM are controlled by the eigenvalues of \cref{eq:2-5}. By assuming that $\omega_{i}$ and $\omega_{i}^{\ast}$ are concentric (truncated) cubes with side lengths $H_{i}$ and $H_{i}^{\ast}$,  respectively, we have the following nearly exponential decay rate for the local approximation errors.
\begin{theorem}[\cite{ma2021novel}]\label{thm:2-2}
For $\epsilon\in (0,\frac{1}{d+1})$, there exists an $N_{\epsilon}>0$, such that for any $n>N_{\epsilon}$,
\begin{equation}
\frac{1}{\lambda^{1/2}_{n+1}} \leq C_{1}e^{2}e^{-n^{(\frac{1}{d+1}-\epsilon)}} e^{-n^{(\frac{1}{d+1}-\epsilon)}R(\rho_{i})}, \quad {\rm if}\; \partial \omega_{i}^{\ast}\cap \partial \Omega_{D} \neq \emptyset,
\end{equation}
and 
\begin{equation}
\frac{1}{\lambda^{1/2}_{n+2}} \leq C_{1}e^{2}e^{-n^{(\frac{1}{d+1}-\epsilon)}} e^{-n^{(\frac{1}{d+1}-\epsilon)}R(\rho_{i})}, \quad {\rm if}\; \partial \omega_{i}^{\ast}\cap \partial \Omega_{D} = \emptyset,
\end{equation}
where $C_{1}$ is the constant given in \cref{eq:1-3}, $R(s) = 1+{s\log(s)}/{(1-s)}$, and $\rho_{i} = H_{i}/H_{i}^{\ast}$.
\end{theorem}

However, in this form it is not sufficient in practice, since the local boundary value problems \cref{eq:2-3,eq:2-4} and the eigenvalue problem \cref{eq:2-5} need to be solved numerically, e.g., using finite element methods. We address this in the next section.

\section{Discrete MS-GFEM}\label{sec-3}
In this section, we present the discrete MS-GFEM based on finite element approximations of the local boundary value problems \cref{eq:2-3,eq:2-4} and of the eigenvalue problem \cref{eq:2-5}.

Let ${\tau}_{h}=\{K\}$ be a shape-regular triangulation of $\Omega$ into triangles (tetrahedrons) or rectangles if $\Omega$ is a rectangular domain, where $h=\max_{K\in {\tau}_{h}}h_{K}$ with $h_{K} = diam (K)$. The mesh-size $h$ is assumed to be small enough to resolve all fine-scale details of the coefficient $A({\bm x})$. Let $V_{h}\subset H^{1}(\Omega)$ be a Lagrange finite element space of lowest order with a basis $\{\varphi_{k}\}_{k=1}^{\mathcal{N}}$, where $\mathcal{N}$ is the dimension of $V_{h}$. The generalization to higher order finite element methods is straightforward. For simplicity, we assume that the oversampling domains $\{ \omega_{i}^{\ast}\}_{i=1}^{M}$ are resolved by the mesh. For each $i=1,\cdots,M$, we define 
\begin{equation}\label{eq:3-1}
\begin{array}{lll}
{\displaystyle  {V}_{h}(\omega^{\ast}_{i}) = \big\{v|_{\omega^{\ast}_{i}}\;:\; v\in V_{h}\big\}, }\\[2mm]
{\displaystyle {V}_{h,D}(\omega^{\ast}_{i}) = \big\{v\in V_{h}(\omega_{i}^{\ast}):\; v = 0 \;\;{\rm on}\;\, \partial \omega^{\ast}_{i} \cap \partial \Omega_{D}\big\},  }\\[2mm]
{\displaystyle V_{h,DI}(\omega^{\ast}_i)= \big\{v\in {V}_{h}(\omega^{\ast}_{i}):\;v = 0 \;\;{\rm on}\;\, (\partial \omega^{\ast}_{i} \cap \partial \Omega_{D})\cup(\partial \omega^{\ast}_{i} \cap \Omega)\big\}, }\\[2mm]
{\displaystyle W_{h}(\omega^{\ast}_i)= \big\{u\in V_{h,D}(\omega^{\ast}_i)\;:\; a_{\omega^{\ast}_{i}}(u,v) = 0,\;\, \forall v\in  V_{h,DI}(\omega^{\ast}_i)\big\}. }
\end{array}
\end{equation}
$V_{h,D}(\omega^{\ast}_{i})$, $V_{h,DI}(\omega^{\ast}_{i})$, and $W_{h}(\omega^{\ast}_i)$ are the finite element discretizations of $H^{1}_{D}(\omega_{i}^{\ast})$, $H^{1}_{DI}(\omega_{i}^{\ast})$, and $H_{A,D}(\omega_{i}^{\ast})$, respectively. Note that $W_{h}(\omega^{\ast}_i) \nsubseteq H_{A,D}(\omega_{i}^{\ast})$.

On each subdomain $\omega_{i}$, the (discrete) local particular function is defined as $u^{p}_{h,i} = \psi_{h,i}|_{\omega_{i}}=(\psi_{h,i}^{r} + \psi_{h,i}^{d})|_{\omega_{i}}$, where $\psi_{h,i}^{r} \in V_{h,DI}(\omega_{i}^{\ast})$ satisfies 
\begin{equation}\label{eq:3-2}
a_{\omega_{i}^{\ast}}(\psi^{r}_{h,i}, v) = F_{\omega_{i}^{\ast}}(v),\quad \forall v\in V_{h,DI}(\omega^{\ast}_i)
\end{equation}
and $\psi_{h,i}^{d} \in V_{h}(\omega_{i}^{\ast})$ satisfies $\psi_{h,i}^{d}=q({\bm x})$ on $\partial \omega_{i}^{\ast}\cap \partial \Omega_D$ and 
\begin{equation}\label{eq:3-3}
a_{\omega_{i}^{\ast}}(\psi^{d}_{h,i}, v) =0,\quad \forall v\in V_{h,D}(\omega^{\ast}_i).
\end{equation}
The (discrete) local approximation space $S_{h,n_i}(\omega_{i})$ is defined as 
\begin{equation}\label{eq:3-4}
S_{h,n_i}(\omega_{i}) = {\rm span}\big\{\phi_{h,1}|_{\omega_{i}},\cdots,\phi_{h,n_{i}}|_{\omega_{i}}\big\},
\end{equation}
where $\{\phi_{h,j}\}_{j=1}^{n_i} \subseteq {W}_{h}(\omega^{\ast}_i)$ are the eigenfunctions corresponding to the $n_{i}$ smallest eigenvalues of the problem:
\begin{equation}\label{eq:3-5}
a_{\omega^{\ast}_{i}}(\phi_{h},v) = \lambda_{h} \,a_{\omega_{i}}(\chi_{i}\phi_{h}, \,\chi_{i}v),\quad \forall \,v\in {W}_{h}(\omega^{\ast}_i).
\end{equation}
The global particular function and the trial space for the discrete MS-GFEM are then defined by
\begin{equation}\label{eq:3-6}
\begin{array}{lll}
{\displaystyle u_{h}^{p} :=\sum_{i=1}^{M}\chi_{i}u^{p}_{h,i}\;\;\;{\rm and}\;\; \; S_{h}(\Omega):= \Big\{ \sum_{i=1}^{M} \chi_{i}v_{i}:\; v_{i}\in S_{h,n_i}(\omega_{i}) \Big\}. }
\end{array}
\end{equation}
The final step of the discrete MS-GFEM is to find the Galerkin approximation $u_{h}^{G}= u_{h}^{p} + u_{h}^{s} \in u_{h}^{p} + S_{h}(\Omega)$ of $u^{e}$ that satisfies
\begin{equation}\label{eq:3-7}
a(u_{h}^{s},\,v) = F(v)-a(u^{p}_{h}, v),\quad \forall v\in S_{h}(\Omega).
\end{equation}
As in the continuous case in \cref{eq:1-2-4}, we have
\begin{equation}\label{eq:3-7-0}
u^{G}_{h} = {\rm argmin}\{\Vert u^{e} - v\Vert_{a}\,:\, v\in u_{h}^{p} + S_{h}(\Omega)\}.
\end{equation}

Furthermore, we introduce the standard finite element approximation of \cref{eq:1-2} on the fine mesh as follows. Find $u^{e}_{h}\in V_{h}$ with $u^{e}_{h}= q({\bm x})$ on $\partial \Omega_{D}$ such that 
\begin{equation}\label{eq:3-8}
a(u^{e}_{h}, v) = F(v),\quad \forall v\in V_{h} \;\,{\rm with}\;\,v = 0 \;\;{\rm on}\;\, \partial \Omega_{D}.
\end{equation}
Similar to the continuous case, the fine-scale solution $u^{e}_{h}$ can be locally decomposed into two orthogonal parts with respect to the inner product $a_{\omega_{i}^{\ast}}(\cdot,\cdot)$ as follows.
\begin{equation}\label{eq:3-10}
u^{e}_{h}|_{\omega_{i}^{\ast}}-\psi_{h,i}\in W_{h}(\omega_{i}^{\ast}),\quad a_{\omega_{i}^{\ast}}(u^{e}_{h}|_{\omega_{i}^{\ast}}-\psi_{h,i},\,\psi_{h,i})=0,\quad \forall \,i=1,\cdots,M.
\end{equation}

\Cref{eq:3-10} can be proved in an identical way as \cref{eq:2-4-0}. As shown in \cite{ma2021novel} for the continuous problem, we will prove that for each $i=1,\cdots,M$, the local approximation space $S_{h,n_i}(\omega_{i})$ defined in \cref{eq:3-4} is optimal for approximating discrete $A$-harmonic functions in an appropriate sense. We first assume that $\partial \omega^{\ast}_{i} \cap \partial \Omega_{D} \neq \emptyset$ and introduce an operator $P_{h,i}:W_{h}(\omega_{i}^{\ast})\rightarrow H_{DI}^{1}(\omega_{i})$ such that
\begin{equation}\label{eq:3-15}
P_{h,i}(v)({\bm x}) = \chi_{i}({\bm x}) v({\bm x})\quad\; \forall {\bm x}\in \omega_{i}.
\end{equation}
$P_{h,i}$ is compact since it is a finite-rank operator. For $n=1,\cdots, {\rm dim}\,(W_{h}(\omega_{i}^{\ast}))$, we consider approximating the set $P_{h,i}(W_{h}(\omega_{i}^{\ast}))$ by $n$-dimensional spaces $Q(n)\subset H_{DI}^{1}(\omega_{i})$. The problem of finding the optimal approximation spaces is formulated as follows. For each $n$, we define the Kolmogorov $n$-width $d_{n}(\omega_{i},\omega_{i}^{\ast})$ of the compact operator $P_{h,i}$ as in \cite{pinkus1985n} by
\begin{equation}\label{eq:3-16}
d_{n}(\omega_{i},\omega_{i}^{\ast}) = \inf_{Q(n)\subset H_{DI}^{1}(\omega_{i})}\sup_{u\in W_{h}(\omega_{i}^{\ast})} \inf_{v\in Q(n)}\frac {\Vert P_{h,i}u-v\Vert_{a,\omega_{i}}}{\Vert u \Vert_{a,\omega_{i}^{\ast}}}.
\end{equation}
Then the optimal approximation space $\widehat{Q}(n)$ satisfies
\begin{equation}\label{eq:3-17}
d_{n}(\omega_{i}, \omega_{i}^{\ast}) =\sup_{u\in W_{h}(\omega_{i}^{\ast})} \inf_{v\in \widehat{Q}(n)}\frac {\Vert P_{h,i}u-v\Vert_{a,\omega_{i}}}{\Vert u \Vert_{a,\omega_{i}^{\ast}}}.
\end{equation}
\cm{
\begin{rem}
If $\partial \omega^{\ast}_{i} \cap \partial \Omega_{D} = \emptyset$, $\Vert\cdot\Vert_{a,\,\omega_{i}^{\ast}}$ is only a seminorm on $W_{h}(\omega_{i}^{\ast})$. In this case, we introduce a subspace of $W_{h}(\omega_{i}^{\ast})$
\begin{equation}\label{eq:3-24}
{\displaystyle \interior{W}_{h}(\omega_{i}^{\ast}) = \big\{u \in W_{h}(\omega_{i}^{\ast}) \,:\, \mathcal{M}_{\omega_{i}}(u) = 0 \big\},}
\end{equation}
where 
\begin{equation}\label{eq:3-25}
\mathcal{M}_{\omega_{i}}(u) =\int_{\omega_{i}} A\nabla(\chi_{i} u) \cdot\nabla \chi_{i}\,d{\bm x}.
\end{equation}
It is not difficult to verify that $\Vert\cdot\Vert_{a,\,\omega_{i}^{\ast}}$ is a norm on $\interior{W}_{h}(\omega_{i}^{\ast})$. We modify the definition of the operator $P_{h,i}$ as 
\begin{equation}\label{eq:3-26}
P_{h,i}:\interior{W}_{h}(\omega_{i}^{\ast})\rightarrow H_{DI}^{1}(\omega_{i})\;\; {\rm such \;that}\;\; P_{h,i}v = \chi_{i}v
\end{equation}
and define the $n$-width accordingly.
\end{rem}
}

The characterization of the $n$-width of a compact operator in Hilbert spaces is well studied. Let $P_{h,i}^{\ast}: H_{DI}^{1}(\omega_{i})\rightarrow W_{h}(\omega_{i}^{\ast})$ ($\interior{W}_{h}(\omega_{i}^{\ast})$) be the adjoint of $P_{h,i}$. We see that $P_{h,i}^{\ast}P_{h,i}: W_{h}(\omega_{i}^{\ast})\rightarrow W_{h}(\omega_{i}^{\ast})$ ($\interior{W}_{h}(\omega_{i}^{\ast})\rightarrow \interior{W}_{h}(\omega_{i}^{\ast})$) is a compact, self-adjoint, positive operator. The following theorem characterizes the $n$-width and the associated optimal approximation space via the singular system of the operator $P_{h,i}$.
\begin{theorem}\label{thm:3-1}
Let $\{ \phi_{h,k}\} $ and $\{\mu_{h,k}\}$ denote the eigenfunctions and eigenvalues of the problem
\begin{equation}\label{eq:3-18}
P_{h,i}^{\ast}P_{h,i}\phi_{h} = \mu_{h} \phi_{h}.
\end{equation}
Then, the optimal approximation space is given by $\widehat{Q}(n) = {\rm span}\{\psi_{h,1},\cdots,\psi_{h,n}\}$, where $\psi_{h,i} = P_{h,i}\phi_{h,i}$ and the $n$-width $d_{n}(\omega_{i},\omega_{i}^{\ast}) = \mu_{h,n+1}^{1/2}$.
\end{theorem}

The $\{ \mu_{h,k}^{1/2}\}$, $\{  \phi_{h,k}\} $, and $\{P_{h,i}\phi_{h,k}\}$ are known as the singular values, as well as the right and left singular vectors of the compact operator $P_{h,i}$. Next we associate the eigenvalue problem \cref{eq:3-5} and the local approximation space $S_{h,n_i}(\omega_{i})$ with the $n$-width.
\cm{
\begin{theorem}\label{thm:3-2}
For each $k=1,\cdots, {\rm dim}\,(W_{h}(\omega_{i}^{\ast}))$, let $(\lambda_{h,k},\,\phi_{h,k})$ be the $k$-th eigenpair of the problem \cref{eq:3-5} (arranged in increasing order). Then, the $n$-width and the associated optimal approximation space are given by
\begin{equation}\label{eq:3-19-0}
d_{n}(\omega_{i}, \omega_{i}^{\ast}) = \lambda^{-1/2}_{h,n+2},\;\; \widehat{Q}(n) = {\rm span}\{\chi_{i}\phi_{h,2},\cdots,\chi_{i}\phi_{h,n+1}\}\quad {\rm if}\;\partial \omega^{\ast}_{i} \cap \partial \Omega_{D} = \emptyset,
\end{equation}
or
\begin{equation}\label{eq:3-19}
d_{n}(\omega_{i}, \omega_{i}^{\ast}) = \lambda^{-1/2}_{h,n+1},\;\; \widehat{Q}(n) = {\rm span}\{\chi_{i}\phi_{h,1},\cdots,\chi_{i}\phi_{h,n}\}\quad {\rm if}\;\partial \omega^{\ast}_{i} \cap \partial \Omega_{D} \neq \emptyset.
\end{equation}
In addition, there exists a $\xi_{i}\in S_{h,n_i}(\omega_{i})={\rm span}\{\phi_{h,1},\cdots,\phi_{h,n_i}\}$ such that
\begin{equation}\label{eq:3-20}
\Vert \chi_{i}(u^{e}_{h} - u^{p}_{h,i} - \xi_{i})\Vert_{a,\omega_{i}}\leq \lambda^{-1/2}_{h,n_{i}+1}\,\Vert u^{e}_{h}\Vert_{a,\omega_{i}^{\ast}}.
\end{equation}
\end{theorem}
\begin{proof}
We only prove the result for the case $\partial \omega^{\ast}_{i} \cap \partial \Omega_{D} = \emptyset$, and the other case can be proved similarly. Let $\lambda_{h} = \mu_{h}^{-1}$ and let $\mathbb{R}$ denote the space of constant functions. We see that ${W}_{h}(\omega^{\ast}_i) = \mathbb{R}\oplus \interior{W}_{h}(\omega_{i}^{\ast})$ and $a_{\omega_{i}^{\ast}}(v,\varphi) = a_{\omega_{i}}(\chi_{i}v,\chi_{i}\varphi) = 0$ for all $v\in \mathbb{R}$ and $\varphi \in \interior{W}_{h}(\omega_{i}^{\ast})$. It follows that the eigenproblem \cref{eq:3-5} can be decoupled into two eigenproblems: one on $\mathbb{R}$ with eigenvalue 0 and another on $\interior{W}_{h}(\omega_{i}^{\ast})$ with positive eigenvalues, i.e.,
\begin{equation}\label{eq:3-20-0}
a_{\omega^{\ast}_{i}}(\phi_{h},v) =  \lambda_{h} \,a_{\omega_{i}}(\chi_{i}\phi_{h}, \,\chi_{i}v)\quad \forall v\in \interior{W}_{h}(\omega_{i}^{\ast}),
\end{equation}
which is the variational formulation of \cref{eq:3-18}. Noting that the $k$-th eigenpair of \cref{eq:3-20-0} is the $(k+1)$-th eigenpair of \cref{eq:3-5}, we get \cref{eq:3-19-0} immediately from \cref{thm:3-1}. 

It remains to prove \cref{eq:3-20}. By \cref{eq:3-10} and the definition of $\interior{W}_{h}(\omega_{i}^{\ast})$, we see that there exists a constant $c\in\mathbb{R}$ such that $u^{e}_{h}-\psi_{h,i}-c\in \interior{W}_{h}(\omega_{i}^{\ast})$. It follows from the definition of the $n$-width and \cref{eq:3-19-0} that there exists a $\widehat{\xi}_{i}\in {\rm span}\{\phi_{h,2},\cdots,\phi_{h,n_{i}}\}$ such that
\begin{equation}\label{eq:3-21}
\begin{array}{lll}
{\displaystyle \Vert \chi_{i}(u^{e}_{h} - \psi_{h,i}- c-\widehat{\xi}_{i})\Vert_{a,\omega_{i}}\leq d_{n_i-1}(\omega_{i},\omega_{i}^{\ast})\Vert u^{e}_{h} -\psi_{h,i}-c\Vert_{a,\omega_{i}^{\ast}} }\\[3mm]
{\displaystyle \qquad \qquad \qquad \qquad \qquad \qquad = \lambda^{-1/2}_{h,n_{i}+1}\,\Vert u^{e}_{h} -\psi_{h,i}\Vert_{a,\omega_{i}^{\ast}}.}
\end{array}
\end{equation}
Define $\xi_{i}=c+\widehat{\xi}_{i}\in S_{h,n_i}(\omega_{i})$. It follows that
\begin{equation}\label{eq:3-22}
\Vert \chi_{i}(u^{e}_{h} - u^{p}_{h,i} - \xi_{i})\Vert_{a,\omega_{i}}\leq \lambda^{-1/2}_{h,n_{i}+1}\,\Vert u^{e}_{h} -\psi_{h,i}\Vert_{a,\omega_{i}^{\ast}}.
\end{equation}
Using \cref{eq:3-10} again, we get
\begin{equation}\label{eq:3-23}
\begin{array}{lll}
{\displaystyle a_{\omega_{i}^{\ast}}(u^{e}_{h}-\psi_{h,i}, u^{e}_{h}-\psi_{h,i}) = a_{\omega_{i}^{\ast}}(u^{e}_{h}-\psi_{h,i}, u^{e}_{h}) \leq \Vert u^{e}_{h}-\psi_{h,i}\Vert_{a,\omega_{i}^{\ast}} \Vert u^{e}_{h}\Vert_{a,\omega_{i}^{\ast}},}
\end{array}
\end{equation}
which implies that $\Vert u^{e}_{h}-\psi_{h,i}\Vert_{a,\omega_{i}^{\ast}}\leq  \Vert u^{e}_{h}\Vert_{a,\omega_{i}^{\ast}}$. Combining this inequality and \cref{eq:3-22} gives \cref{eq:3-20}.
\end{proof}
}

The following theorem gives the error bound of the discrete MS-GFEM.
\begin{theorem}\label{thm:3-3}
Let $u^{e}$ and $u^{e}_{h}$ be the solutions of \cref{eq:1-2,eq:3-8}, respectively, and let $u_{h}^{G}$ be the discrete MS-GFEM approximation. Then,
\begin{equation}\label{eq:3-28}
\displaystyle \big\Vert u^{e} - u^{G}_{h} \big\Vert_{a} \leq  \Vert u^{e} - u^{e}_{h} \Vert_{a} +\sqrt{\kappa\kappa^{\ast}}\big(\max_{i=1,\cdots,M}\lambda^{-1/2}_{h,n_{i}+1}\big)\Vert u^{e}_{h}\Vert_{a},
\end{equation}
where $\lambda_{h,n_{i}+1}$ is the $(n_i+1)$-th eigenvalue of the problem \cref{eq:3-5}.
\end{theorem}
\begin{proof}
By \cref{thm:3-2}, there exists a $\xi_{i}\in S_{h,n_{i}}(\omega_{i})$, $i=1,\cdots,M$, such that
\begin{equation}\label{eq:3-29}
\big\Vert \chi_{i}(u^{e}_{h}-u^{p}_{h,i}-\xi_{i})\big\Vert_{a,\,\omega_{i}}\leq \lambda^{-1/2}_{h,n_{i}+1}\Vert u^{e}_{h}\Vert_{a,\,\omega^{\ast}_{i}}.
\end{equation}
Let $\Psi_{h} = u^{p}_{h} + \sum_{i=1}^{M}\chi_{i}\xi_{i}\in u^{p}_{h}+S_{h}(\Omega)$. Then
\begin{equation}\label{eq:3-30}
\begin{array}{lll}
{\displaystyle \quad \big\Vert u^{e}_{h} - \Psi_{h} \big\Vert^{2}_{a} =\Big\Vert \sum_{i=1}^{M}\chi_{i}(u^{e}_{h}-u^{p}_{h,i}-\xi_{i})\Big\Vert^{2}_{a} }\\[3mm]
{\displaystyle \leq \kappa \sum_{i=1}^{M}\big\Vert \chi_{i}(u^{e}_{h}-u^{p}_{h,i}-\xi_{i})\big\Vert^{2}_{a,\,\omega_{i}} \leq \kappa\sum_{i=1}^{M}\lambda^{-1}_{h,n_{i}+1}\Vert u^{e}_{h}\Vert^{2}_{a,\,\omega^{\ast}_{i}}  }\\[3mm]
{\displaystyle \leq \kappa \big(\max_{i=1,\cdots,M}\lambda^{-1}_{h,n_{i}+1}\big)\sum_{i=1}^{M}\Vert u^{e}_{h}\Vert^{2}_{a,\,\omega^{\ast}_{i}} \leq \kappa\kappa^{\ast}\big(\max_{i=1,\cdots,M}\lambda^{-1}_{h,n_{i}+1}\big)\Vert u^{e}_{h}\Vert^{2}_{a}, }
\end{array}
\end{equation}
where we have used the assumption that each point ${\bm x}\in\Omega$ belongs to at most $\kappa^{\ast}$ subdomains $\omega_{i}^{\ast}$ as in \cref{thm:1-0}. Combining the triangle inequality and \cref{eq:3-30} gives
\begin{equation}\label{eq:3-31}
 \big\Vert u^{e} - \Psi_{h} \big \Vert_{a} \leq \Vert u^{e} - u^{e}_{h} \Vert_{a} +\sqrt{\kappa\kappa^{\ast}}\big(\max_{i=1,\cdots,M}\lambda^{-1/2}_{h,n_{i}+1}\big)\Vert u^{e}_{h}\Vert_{a}.
\end{equation}
The desired estimate \cref{eq:3-28} follows from \cref{eq:3-31,eq:3-7-0}.
\end{proof}

\Cref{thm:3-3} indicates that the error of the discrete MS-GFEM is bounded by the FE error in $V_{h}$ and the errors arising from the local approximations of the fine-scale solution. By implementing the discrete MS-GFEM on a fine mesh such that the error resulting from the fine-scale discretization is sufficiently small, we only need to concern ourselves with the local approximation errors that are controlled by the eigenvalues of \cref{eq:3-5}. In the next section, we will prove the convergence of the eigenvalues of the discrete problem \cref{eq:3-5} towards the eigenvalues of the continuous problem \cref{eq:2-5} as $h\rightarrow 0$ and also establish a nearly exponential decay rate for the local approximation errors.

\subsection{Mixed FE approximations and Caccioppoli-type inequalities}
In this subsection, we present some preliminary results that constitute the crucial technical tools for proving the convergence of the discrete MS-GFEM in the following section. We assume that $\omega$ and $\omega^{\ast}$ are open subsets of $\Omega$ with $\omega \subset \omega^{\ast}\subset \Omega$ throughout this subsection.

\begin{lemma}\label{lem:3-s2-1}
Assume that $\partial \omega^{\ast} \cap \partial \Omega_{D} \neq \emptyset$ and that $\mathcal{F}$ is a bounded linear functional on $H^{1}_{D}(\omega^{\ast})$. Consider the problem of finding $z\in H^{1}_{D}(\omega^{\ast})$ and $p\in H^{1}_{DI}(\omega^{\ast})$ such that
\begin{equation}\label{eq:3-s2-1}
\begin{aligned}
a_{\omega^{\ast}}(z,\phi) + a_{\omega^{\ast}}(\phi,p) =&\, \mathcal{F}(\phi)\quad \forall \phi\in H^{1}_{D}(\omega^{\ast}),\\
a_{\omega^{\ast}}(z,\xi) =&\;0\qquad\;\; \forall \xi\in H^{1}_{DI}(\omega^{\ast}).
\end{aligned}
\end{equation}
Then there exists a unique solution $(z,p)$ to \cref{eq:3-s2-1} and 
\begin{equation}\label{eq:3-s2-2}
\Vert z\Vert_{a,\omega^{\ast}} + \Vert p\Vert_{a,\omega^{\ast}}\leq C\Vert \mathcal{F}\Vert_{(H^{1}_{D}(\omega^{\ast}))^{\prime}}.
\end{equation}
\end{lemma}
\begin{proof}
Since $\partial \omega^{\ast} \cap \partial \Omega_{D} \neq \emptyset$, the bilinear form $a_{\omega^{\ast}}(\cdot,\cdot)$ is coercive on $H^{1}_{D}(\omega^{\ast})$ and $\Vert\cdot\Vert_{a,\omega^{\ast}}$ is a norm on $H^{1}_{D}(\omega^{\ast})$ and $H^{1}_{DI}(\omega^{\ast})$. Further, it is easy to verify that the inf--sup condition on $a_{\omega^{\ast}}(\cdot,\cdot)$ is satisfied. By Theorem 4.2.3 of \cite{boffi2013mixed}, the problem \cref{eq:3-s2-1} has a unique solution satisfying \cref{eq:3-s2-2}.
\end{proof}

\begin{lemma}\label{lem:3-s2-2}
Assume that $\partial \omega^{\ast}\cap \partial \Omega_{D} \neq \emptyset$. Let $\mathcal{F}$ be a bounded linear functional on $H^{1}_{D}(\omega^{\ast})$ and let $(z,p)$ satisfy \cref{eq:3-s2-1}. Consider the discrete problem of finding $z_{h}\in V_{h,D}(\omega^{\ast})$ and $p_{h}\in V_{h,DI}(\omega^{\ast})$ such that 
\begin{equation}\label{eq:3-s2-3}
\begin{aligned}
a_{\omega^{\ast}}(z_{h},\phi_{h}) + a_{\omega^{\ast}}(\phi_{h},p_{h}) =&\, \mathcal{F}(\phi_{h})\quad \forall \phi_{h}\in V_{h,D}(\omega^{\ast}),\\
a_{\omega^{\ast}}(z_{h},\xi_{h}) =&\;0\qquad\,\quad \forall \xi_{h}\in V_{h,DI}(\omega^{\ast}).
\end{aligned}
\end{equation}
Then there exists a unique solution $(z_h,p_h)$ to \cref{eq:3-s2-3}. There is also a constant $C$ independent of $h$, $(z,p)$, and $(z_h,p_h)$ such that
\begin{equation}\label{eq:3-s2-4}
\begin{array}{lll}
{\displaystyle \qquad \Vert z-z_h\Vert_{a,\omega^{\ast}} + \Vert p-p_h\Vert_{a,\omega^{\ast}}}\\[3mm]
{\displaystyle \leq C\big(\inf_{v_h\in V_{h,D}(\omega^{\ast})}\Vert z-v_h\Vert_{a,\omega^{\ast}} + \inf_{q_h\in V_{h,DI}(\omega^{\ast})}\Vert p-q_h\Vert_{a,\omega^{\ast}}\big).}
\end{array} 
\end{equation}
\end{lemma}
\begin{proof}
Since $\partial \omega^{\ast} \cap \partial \Omega_{D} \neq \emptyset$, the bilinear form $a_{\omega^{\ast}}(\cdot,\cdot)$ is coercive on $V_{h,D}(\omega^{\ast})$ and the discrete inf--sup condition on $a_{\omega^{\ast}}(\cdot,\cdot)$ is satisfied. Hence, it follows from Theorem 5.2.5 of \cite{boffi2013mixed} that the discrete problem \cref{eq:3-s2-3} has a unique solution satisfying \cref{eq:3-s2-4}.
\end{proof}
\begin{rem}\label{rem:3-1}
If $\partial \omega^{\ast}\cap \partial \Omega_{D} = \emptyset$, we further assume that the linear functional $\mathcal{F}$ satisfies $\mathcal{F}(c)=0$ for any $c\in\mathbb{R}$. Then \cref{lem:3-s2-1,lem:3-s2-2} still hold in this case with $H^{1}_{D}(\omega^{\ast})$ and $V_{h,D}(\omega^{\ast})$ replaced by $H^{1}(\omega^{\ast})/\mathbb{R}$ and $V_{h}(\omega^{\ast})/\mathbb{R}$, respectively.
\end{rem}

Next we give the continuous and discrete Caccioppoli-type inequalities which are essential for studying the continuous and discrete $A$-harmonic spaces.
\begin{lemma}[\cite{ma2021novel}]\label{lem:3-s2-3}
Assume that $\eta \in W^{1,\infty}(\omega^{\ast})$ satisfies $\eta({\bm x}) = 0$ on $\partial\omega^{\ast} \cap \Omega$. Then, for any $u,\,v\in H_{A,D}(\omega^{\ast})$,
\begin{equation}\label{eq:3-s2-5}
\int_{\omega^{\ast}} A \nabla (\eta u)\cdot \nabla (\eta v)\,d{\bm x} = \int_{\omega^{\ast}}(A\nabla \eta \cdot \nabla \eta) uv\,d{\bm x}.
\end{equation}
In particular,
\begin{equation}\label{eq:3-s2-5-0}
\Vert \eta u \Vert_{a, \omega^{\ast}} \leq \beta^{\frac12} \Vert \nabla \eta \Vert_{L^{\infty}(\omega^{\ast})} \Vert u \Vert_{L^{2}(\omega^{\ast})},\quad \forall u\in H_{A,D}(\omega^{\ast}),
\end{equation}
where $\beta$ is defined in \cref{eq:1-1-0}.
\end{lemma}
\begin{rem}
\Cref{eq:3-s2-5} implies that the continuous eigenvalue problem \cref{eq:2-5} can be rewritten as 
\begin{equation}\label{eq:3-s2-5-1}
a_{\omega_{i}^{\ast}}(\phi, v) = \lambda\,\big(A\nabla\chi_{i}\cdot\nabla\chi_{i} \,\phi, \,v\big)_{L^{2}(\omega^{\ast}_{i})},\quad \forall v\in H_{A,D}(\omega_{i}^{\ast}).
\end{equation}
Similar to \cref{eq:3-15}, we can define an operator $P_{i}: H_{A,D}(\omega_{i}^{\ast})\rightarrow H_{0}^{1}(\omega_{i})$ as 
\begin{equation}\label{eq:3-s2-5-2}
P_{i}(v)({\bm x}) = \chi_{i}({\bm x}) v({\bm x})\quad\; \forall {\bm x}\in \omega_{i}.
\end{equation}
\Cref{eq:3-s2-5-0} implies that $P_{i}$ is a compact operator.
\end{rem}

To bound the FE error in the discrete MS-GFEM, we use the following result from \cite{demlow2011local}.
\begin{lemma}\label{lem:3-s2-4}
Let $\omega\subset\omega^{\ast}$ be open connected sets with $\delta:={dist}\,\big(\omega, \, \partial \omega^{\ast}/\partial \Omega\big)>0$. In addition, let $\max_{K\cap \omega^{\ast}\neq \emptyset}h_{K}\leq\frac{1}{4}\delta$. Then, for each $u_{h}\in W_{h}(\omega^{\ast})$,
\begin{equation}\label{eq:3-s2-14}
\Vert u_{h}\Vert_{a,\omega}\leq C\delta^{-1}\Vert u_{h} \Vert_{L^{2}(\omega^{\ast})},
\end{equation}
where $C$ depends only on $\alpha$, $\beta$, $d$, and the shape regularity of the mesh $\tau_{h}$.
\end{lemma}
\begin{rem}
The proof of \cref{lem:3-s2-4} in \cite{demlow2011local} relies critically on the following $\mathit{superapproximation}$ result: Let $S^{r}_{h}$ be a finite element space consisting of continuous piecewise polynomials of degree $r$ and let $\eta\in W^{r+1,\infty}(\Omega)$ with $|\eta|_{W^{j,\infty}(\Omega)}\leq C\delta^{-j}$ for integers $0\leq j\leq r+1$. Denote by $I_{h}:C(\Omega)\rightarrow S^{r}_{h}$ the standard Lagrange interpolant. Then for each $u_{h}\in S^{r}_{h}$ and each $K\in \tau_{h}$ satisfying $h_{K}\leq \delta$,
\begin{equation}\label{eq:3-s2-15}
\Vert \eta^{2}u_{h}-I_{h}(\eta^{2}u_{h})\Vert_{H^{1}(K)}\leq C\Big(\frac{h_{K}}{\delta}\Vert\nabla(\eta u_{h})\Vert_{L^{2}(K)} + \frac{h_{K}}{\delta^{2}}\Vert u_{h} \Vert_{L^{2}(K)}\Big).
\end{equation}
The superapproximation estimate \cref{eq:3-s2-15} holds for standard Lagrange finite element spaces defined on shape-regular simplicial meshes and for tensor-product finite element spaces defined on rectangular meshes. For more superapproximation results, we refer to \cite{nitsche1974interior,bramble1975maximum,wahlbin2006superconvergence}.
\end{rem}

\begin{corollary}\label{cor:3-1}
Let $\omega$ and $\omega^{\ast}$ satisfy the same assumptions as in \cref{lem:3-s2-4}. Assume that $\eta \in W^{1,\infty}(\omega^{\ast})$ with $supp\,(\eta)\subset \overline{\omega}$. Then, for each $u_{h}\in W_{h}(\omega^{\ast})$,
\begin{equation}\label{eq:3-s2-25}
\Vert \eta u_{h}\Vert_{a,\omega^{\ast}}\leq \big(C\delta^{-1} \Vert \eta\Vert_{L^{\infty}(\omega)} + \beta^{1/2}\Vert \nabla \eta\Vert_{L^{\infty}(\omega)} \big)\Vert u_{h} \Vert_{L^{2}(\omega^{\ast})},
\end{equation}
where $C$ is the same constant as in \cref{eq:3-s2-14}.
\end{corollary}
\begin{proof}
Using the chain rule, the triangle inequality, and the assumptions on $\eta$, we have
\begin{equation}\label{eq:3-s2-26}
\begin{array}{lll}
{\displaystyle \Vert \eta u_{h} \Vert_{a, \omega^{\ast}}\leq \Big(\int_{\omega}(A\nabla u_{h} \cdot \nabla u_{h}) \eta^{2}\,d{\bm x}\Big)^{\frac12} +  \Big(\int_{\omega}(A\nabla \eta \cdot \nabla \eta) u_{h}^{2}\,d{\bm x}\Big)^{\frac12} }\\[4mm]
{\displaystyle \qquad \qquad\;\; \leq \Vert \eta\Vert_{L^{\infty}(\omega)}\Vert u_{h}\Vert_{a,\omega} + \beta^{1/2}\Vert \nabla \eta\Vert_{L^{\infty}(\omega)} \Vert u_{h} \Vert_{L^{2}(\omega)}. }
\end{array}
\end{equation}
\Cref{eq:3-s2-25} follows immediately from \cref{eq:3-s2-26} and \cref{lem:3-s2-4}.
\end{proof}

In general, the equation \cref{eq:3-s2-5} does not hold for discrete $A$-harmonic functions. However, we have
\begin{lemma}\label{lem:3-s2-5}
Assume that $\eta \in W^{2,\infty}(\omega^{\ast})$ satisfies $\eta({\bm x}) = 0$ on $\partial\omega^{\ast} \cap \Omega$. Then there exists a constant $C$ independent of $h$, such that for any $u_{h}$, $v_{h}\in W_{h}(\omega^{\ast})$,
\begin{equation}\label{eq:3-s2-27}
\Big|\int_{\omega^{\ast}} A \nabla (\eta u_{h})\cdot \nabla (\eta v_{h})\,d{\bm x} - \int_{\omega^{\ast}}(A\nabla \eta \cdot \nabla \eta) u_{h}v_{h}\,d{\bm x}\Big|\leq Ch \Vert u_{h}\Vert_{H^{1}(\omega^{\ast})} \Vert v_{h}\Vert_{H^{1}(\omega^{\ast})}.
\end{equation}
\end{lemma}
\begin{proof}
For any $u_{h}$, $v_{h}\in V_{h}(\omega^{\ast})$, a direct calculation gives
\begin{equation}\label{eq:3-s2-28}
\begin{array}{lll}
{\displaystyle \int_{\omega^{\ast}} A\nabla (\eta u_{h})\cdot \nabla (\eta v_{h})\,d{\bm x} = \int_{\omega^{\ast}}(A\nabla \eta \cdot \nabla \eta) u_{h}v_{h}\,d{\bm x} - \int_{\omega^{\ast}}(A\nabla u_{h} \cdot \nabla \eta) \eta v_{h}\,d{\bm x}}\\[4mm]
{\displaystyle \qquad \qquad +\,\int_{\omega^{\ast}}(A\nabla \eta \cdot \nabla v_{h}) \eta u_{h}\,d{\bm x} + \int_{\omega^{\ast}}A\nabla u_{h} \cdot \nabla (\eta^{2}v_{h})\,d{\bm x}.}
\end{array}
\end{equation}
Exchanging $u_{h}$ and $v_{h}$, we can get a similar equation as \cref{eq:3-s2-28}. Adding the two equations together and using the symmetry of the matrix $A$, it follows that
\begin{equation}\label{eq:3-s2-29}
\begin{array}{lll}
{\displaystyle \int_{\omega^{\ast}} A \nabla (\eta u_h)\cdot \nabla (\eta v_h)\,d{\bm x} - \int_{\omega^{\ast}}(A\nabla \eta \cdot \nabla \eta) u_h v_h\,d{\bm x} }\\[3mm]
{\displaystyle  = \frac{1}{2}\Big(\int_{\omega^{\ast}}A\nabla u_h \cdot \nabla (\eta^{2}v_h)\,d{\bm x} \,+\, \int_{\omega^{\ast}}A\nabla v_h \cdot \nabla (\eta^{2}u_h)\,d{\bm x}\Big).}
\end{array}
\end{equation}
Now assume that $u_h$, $v_h\in W_{h}(\omega^{\ast})$. Then, we have $a_{\omega^{\ast}}(u_h,\,I_{h}(\eta^{2}v_h)) = 0$, where $I_{h}$ is the standard Lagrange interpolant. Using the superapproximation estimate \cref{eq:3-s2-15}, we obtain
\begin{equation}\label{eq:3-s2-30}
\begin{array}{lll}
{\displaystyle |a_{\omega^{\ast}}(u_h,\,\eta^{2}v_h)| =|a_{\omega^{\ast}}(u_h,\,\eta^{2}v_h-I_{h}(\eta^{2}v_h))| }\\[2mm]
{\displaystyle \leq \Vert u_h\Vert_{a,\omega^{\ast}} \Vert \eta^{2}v_h-I_{h}(\eta^{2}v_h)\Vert_{a,\omega^{\ast}} \leq Ch \Vert u_h\Vert_{H^{1}(\omega^{\ast})} \Vert v_h\Vert_{H^{1}(\omega^{\ast})}.}
\end{array}
\end{equation}
Similarly, we have $|a_{\omega^{\ast}}(v_h,\,\eta^{2}u_h)|\leq Ch \Vert u_h\Vert_{H^{1}(\omega^{\ast})} \Vert v_h\Vert_{H^{1}(\omega^{\ast})}$. Combining these two inequalities with \cref{eq:3-s2-29}, we get \cref{eq:3-s2-27}.
\end{proof}
\begin{rem}
The assumption on $\eta$ can be weakened. In fact, we only need $\eta\in W^{1,\infty}(\omega^{\ast})$ and that $\eta\in W^{2,\infty}(K)$ for each $K\in \tau_{h}$ in the proof. \Cref{eq:3-s2-27} shows that the discrete eigenvalue problem \cref{eq:3-5} differs from the problem
\begin{equation}
a_{\omega_{i}^{\ast}}(\phi_{h}, v_h) = \lambda_{h}\,\big(A\nabla\chi_{i}\cdot\nabla\chi_{i} \,\phi_{h}, \,v_h\big)_{L^{2}(\omega^{\ast}_{i})},\quad \forall v_h\in W_{h}(\omega^{\ast})
\end{equation}
by an order of $O(h)$.
\end{rem}

\section{Convergence analysis}\label{sec-4}
We have shown in \cref{thm:3-2} that the local approximation error in each subdomain is bounded by the $n$-width of the operator $P_{h,i}$, or equivalently, the square root of the $(n_{i}+1)$-th eigenvalue of \cref{eq:3-5}. In this section, we justify our method theoretically in two ways. First we prove that as $h\rightarrow 0$, the eigenvalues of the discrete problem \cref{eq:3-5} converge to the eigenvalues of the continuous problem \cref{eq:2-5}, which have been shown to decay nearly exponentially. Consequently, the error bound of the discrete MS-GFEM converges to that of the continuous MS-GFEM as $h\rightarrow 0$; see Remark~\ref{rem:4-0}. Next we prove directly that the $n$-width of the operator $P_{h,i}$ decays nearly exponentially if $h$ is sufficiently small, thus guaranteeing good approximation properties of the discrete MS-GFEM also for fixed $h$. To simplify the notation, we omit the subscript index $i$ of subdomains in this section.
%
%
%
%
%
%
%
%

\subsection{Convergence of the eigenvalues}
\cm{We first formulate the continuous and discrete eigenproblems \cref{eq:2-5,eq:3-5} as  spectral problems of compact operators. If $\partial \omega^{\ast}\cap\partial \Omega_{D}\neq \emptyset$, recalling the definition of the compact operator $P$ in \cref{eq:3-s2-5-2}, we define the operator $T=P^{\ast}P:H_{A,D}(\omega^{\ast})\rightarrow H_{A,D}(\omega^{\ast})$ such that for each $u\in H_{A,D}(\omega^{\ast})$, $Tu\in H_{A,D}(\omega^{\ast})$ satisfies
\begin{equation}\label{eq:4-1}
a_{\omega^{\ast}}(Tu, v) = a_{\omega}(\chi u, \chi v),\quad \forall v\in H_{A,D}(\omega^{\ast}).
\end{equation}
Similarly, we define the discrete operator $T_{h}=P_{h}^{\ast}P_{h}:W_{h}(\omega^{\ast})\rightarrow W_{h}(\omega^{\ast})$ such that for each $u\in W_{h}(\omega^{\ast})$, $T_{h}u\in W_{h}(\omega^{\ast})$ satisfies
\begin{equation}\label{eq:4-2}
a_{\omega^{\ast}}(T_{h}u, v) = a_{\omega}(\chi u, \chi v),\quad \forall v\in W_{h}(\omega^{\ast}).
\end{equation}
If $\partial \omega^{\ast}\cap\partial \Omega_{D} = \emptyset$, we define the operator $T$ on a subspace of ${H}_{A,D}(\omega^{\ast})$, i.e., 
\begin{equation}\label{eq:4-33-0}
\interior{H}_{A,D}(\omega^{\ast}) = \big\{u \in  {H}_{A,D}(\omega^{\ast})\,:\, \mathcal{M}_{\omega}(u) = 0 \big\}
\end{equation}
with the operator $\mathcal{M}_{\omega}(\cdot)$ defined in \cref{eq:3-25}, such that for each $u\in \interior{H}_{A,D}(\omega^{\ast})$, $Tu\in \interior{H}_{A,D}(\omega^{\ast})$ satisfies
\begin{equation}\label{eq:4-1-0}
a_{\omega^{\ast}}(Tu, v) = a_{\omega}(\chi u, \chi v),\quad \forall v\in \interior{H}_{A,D}(\omega^{\ast}).
\end{equation}
Similarly, the operator $T_{h}$ is defined on $\interior{W}_{h}(\omega^{\ast})$ such that for each $u\in \interior{W}_{h}(\omega^{\ast})$, $T_{h}u\in \interior{W}_{h}(\omega^{\ast})$ satisfies
\begin{equation}\label{eq:4-2-0}
a_{\omega^{\ast}}(T_{h}u, v) = a_{\omega}(\chi u, \chi v),\quad \forall v\in \interior{W}_{h}(\omega^{\ast}),
\end{equation}
where $\interior{W}_{h}(\omega^{\ast})$ is defined in \cref{eq:3-24}.

Note that $T$ and $T_{h}$ ($0<h\leq 1$) are positive, self-adjoint, and compact operators. Consider the spectral problem for the operator $T$
\begin{equation}\label{eq:4-3}
\begin{array}{lll}
{\displaystyle u_{k}\in H_{A,D}(\omega^{\ast})\,\,({\it resp}.\,\interior{H}_{A,D}(\omega^{\ast})), \quad Tu_{k} = \mu_{k} u_{k},\quad k=1,2,\cdots,}\\[2mm]
{\displaystyle \mu_{1}\geq \mu_{2} \geq\cdots\mu_{k}\geq\cdots,\quad \mu_{k}>0,}
\end{array}
\end{equation}
and the spectral problems for $T_{h}$
\begin{equation}\label{eq:4-4}
\begin{array}{lll}
{\displaystyle u_{h,k}\in W_{h}(\omega^{\ast})\,\,({\it resp}.\,\interior{W}_{h}(\omega^{\ast})), \quad T_{h}u_{h,k} = \mu_{h,k} u_{h,k},\quad k=1,2,\cdots,}\\[2mm]
{\displaystyle \mu_{h,1}\geq \mu_{h,2} \geq\cdots\mu_{h,k}\geq\cdots,\quad \mu_{h,k}>0,}
\end{array}
\end{equation}
where the eigenvalues are enumerated in a nonascending order and repeated according to their multiplicities. It is easy to see that the eigenproblems \cref{eq:2-5,eq:3-5} are equivalent to \cref{eq:4-3,eq:4-4} in the case that $\partial \omega^{\ast}\cap\partial \Omega_{D}\neq \emptyset$ with $\mu_{k}=1/\lambda_{k}$ and $\mu_{h,k} = 1/\lambda_{h,k}$, respectively. If $\partial \omega^{\ast}\cap\partial \Omega_{D}= \emptyset$, \cref{eq:2-5,eq:3-5} are equivalent to \cref{eq:4-3,eq:4-4} up to a constant eigenfunction, respectively; see the proof of \cref{thm:3-2}.}

We aim to prove that for each $k\in \mathbb{N}$, $\mu_{h,k} \rightarrow \mu_{k}$ as $h\rightarrow 0$. This means that all the eigenvalues of \cref{eq:4-3} are well approximated and there are no spurious eigenvalues that pollute the spectrum. Note that \cref{eq:4-4} is a non-conforming approximation of \cref{eq:4-3} since $W_{h}(\omega^{\ast}) \nsubseteq H_{A,D}(\omega^{\ast})$. 

To prove the convergence of eigenvalues, we adapt an abstract framework to our problem that has been developed in \cite[Chapter 11]{jikov2012homogenization} for the investigation of spectral problems in homogenization theory. To the best of our knowledge, this is the first time that this abstract framework is used to study FE approximations of eigenvalue problems.

\begin{assumption}\label{ass:2} The spaces $W_{h}(\omega^{\ast})$, $H_{A,D}(\omega^{\ast})\;({\it resp}.\,\interior{W}_{h}(\omega^{\ast}),\,\interior{H}_{A,D}(\omega^{\ast}))$ and the operators $T_{h}$, $T$ satisfy the following conditions. (For ease of notation, we only give the conditions for the case that $\partial \omega^{\ast}\cap\partial \Omega_{D}\neq \emptyset$).

\vspace{3mm}
{\rm I}. \emph{There exist continuous linear operators $R_{h}:H_{A,D}(\omega^{\ast})\rightarrow W_{h}(\omega^{\ast})$ such that
\begin{equation}\label{eq:4-5}
\Vert R_{h}u\Vert_{a,\omega^{\ast}}\leq c_{0} \Vert u\Vert_{a,\omega^{\ast}}, \quad \forall u\in H_{A,D}(\omega^{\ast}),
\end{equation}
\hspace{1cm} where the constant $c_{0}$ is independent of $h$; moreover, 
\begin{equation}\label{eq:4-6}
\lim_{h\rightarrow0}a_{\omega^{\ast}}(u_{h}, v_{h}) = a_{\omega^{\ast}}(u, v),
\end{equation}
\hspace{1cm}provided that}
\begin{equation}\label{eq:4-7}
\begin{array}{lll}
{\displaystyle \lim_{h\rightarrow 0}\Vert u_{h} - R_{h}u\Vert_{a,\omega^{\ast}} = 0,\quad \lim_{h\rightarrow 0}\Vert v_{h} - R_{h}v\Vert_{a,\omega^{\ast}} = 0 }\\[4mm]
{\displaystyle \qquad \quad u_{h},\, v_{h}\in W_{h}(\omega^{\ast}),\quad u,\,v\in H_{A,D}(\omega^{\ast}).}
\end{array}
\end{equation}

{\rm II}. \emph{The operators $T_{h}$, $T$ are positive, compact and self-adjoint; and the norms} 

\emph{$\quad \,\;\Vert T_{h}\Vert=\Vert T_{h}\Vert_{\mathcal{L}(W_{h}(\omega^{\ast}))}$ are bounded by a constant independent of $h$.}

\vspace{3mm}
{\rm III}. \emph{If $\psi_{h}\in W_{h}(\omega^{\ast})$, $\psi\in H_{A,D}(\omega^{\ast})$ and 
\begin{equation}\label{eq:4-8}
\lim_{h\rightarrow 0}\Vert \psi_{h} - R_{h}\psi \Vert_{a,\omega^{\ast}} = 0,
\end{equation}
\hspace{1.3cm}then
\begin{equation}\label{eq:4-9}
\lim_{h\rightarrow 0}\Vert T_{h}\psi_{h} - R_{h}T\psi \Vert_{a,\omega^{\ast}} = 0.
\end{equation}
}

{\rm IV}. \emph{For any sequence $\psi_{h}\in W_{h}(\omega^{\ast})$ such that $\sup_{h\in (0,1]}\Vert \psi_{h} \Vert_{a,\omega^{\ast}}<\infty$, there} 

\hspace{0.65cm}\emph{exists a subsequence $\psi_{h^{\prime}}$ and a function $u\in H_{A,D}(\omega^{\ast})$ such that 
\begin{equation}\label{eq:4-10}
\Vert T_{h^{\prime}}\psi_{h^{\prime}} - R_{h^{\prime}}u \Vert_{a,\omega^{\ast}}\rightarrow 0\quad as \quad h^{\prime}\rightarrow 0.
\end{equation}
}
\end{assumption}

By Lemma 11.3 and Theorem 11.4 of \cite{jikov2012homogenization}, we have
\begin{theorem}\label{thm:4-1}
Let the spaces $W_{h}(\omega^{\ast})$, $H_{A,D}(\omega^{\ast})\;({\it resp}.\,\interior{W}_{h}(\omega^{\ast}),\,\interior{H}_{A,D}(\omega^{\ast}))$ and the operators $T_{h}$, $T$ satisfy conditions {\rm I} - {\rm IV} in \cref{ass:2}. Then
\begin{equation}\label{eq:4-11}
\mu_{h,k}\rightarrow\mu_{k},\quad k=1,2,\cdots, \;\;{\rm as}\;\;h\rightarrow 0,
\end{equation}
where $\mu_{k}$ and $\mu_{h,k}$ are eigenvalues of problems \cref{eq:4-3,eq:4-4}, respectively. Moreover, for sufficiently small $h$,
\begin{equation}\label{eq:4-12}
|\mu_{h,k} - \mu_{k}|\leq 2\sup_{\substack{u\in N(\mu_k, T),\\ \Vert u\Vert_{a,\omega^{\ast}}=1}}\Vert T_{h}R_{h} u - R_{h}Tu\Vert_{a,\omega^{\ast}},\quad k=1,2,\cdots,
\end{equation}
where $N(\mu_k, T)$ is the eigenspace of operator $T$ corresponding to the eigenvalue $\mu_{k}$:
\begin{equation}\label{eq:4-13}
N(\mu_k, T) = \big\{u\in H_{A,D}(\omega^{\ast})\,\,(\interior{H}_{A,D}(\omega^{\ast}))\;:\; Tu = \mu_{k}u \big\}.
\end{equation}
\end{theorem}
\begin{rem}\label{rem:4-0}
\Cref{eq:4-11} indicates that $1/\lambda_{h,k}\rightarrow 1/\lambda_{k}$ as $h\rightarrow 0$, where $\lambda_{h,k}$ and $\lambda_{k}$ are the eigenvalues of the discrete problem \cref{eq:3-5} and the continuous problem \cref{eq:2-5}, respectively. Combining \cref{thm:4-1,thm:3-3} gives that 
\begin{equation}
\begin{array}{lll}
{\displaystyle \big\Vert u^{e} - u^{G}_{h} \big\Vert_{a} \leq  \Vert u^{e} - u^{e}_{h} \Vert_{a} +\sqrt{\kappa\kappa^{\ast}}\big(\max_{i=1,\cdots,M}\lambda^{-1/2}_{h,n_{i}+1}\big)\Vert u^{e}_{h}\Vert_{a}  }\\[3mm]
{\displaystyle \qquad \qquad \qquad \longrightarrow \sqrt{\kappa\kappa^{\ast}}\big(\max_{i=1,\cdots,M}\lambda^{-1/2}_{n_{i}+1}\big)\Vert u^{e}\Vert_{a} \quad {\rm as}\;\;h\rightarrow 0.}
\end{array}
\end{equation}
Therefore, the error bound of the discrete MS-GFEM converges to that of the continuous MS-GFEM; see \cref{thm:1-0,thm:2-1}.
\end{rem}
\begin{proof}
In what follows, we check the validity of conditions I - IV. We first assume that $\partial \omega^{\ast}\cap\partial \Omega_{D}\neq \emptyset$ and introduce the $H^{1}_{D}(\omega^{\ast})$ orthogonal projection, denoted by $\Pi_{h}:H^{1}_{D}(\omega^{\ast})\rightarrow V_{h,D}(\omega^{\ast})$, such that for each $u\in H^{1}_{D}(\omega^{\ast})$, $\Pi_{h}u\in V_{h,D}(\omega^{\ast})$ satisfies
\begin{equation}\label{eq:4-14}
a_{\omega^{\ast}}(\Pi_{h}u,v) = a_{\omega^{\ast}}(u,v),\quad \forall v\in V_{h,D}(\omega^{\ast}).
\end{equation}
If $u\in H_{A,D}(\omega^{\ast})$, it follows that $a_{\omega^{\ast}}(\Pi_{h}u,v) = a_{\omega^{\ast}}(u,v)=0$ for any $v\in V_{h,DI}(\omega^{\ast})$, which implies that $\Pi_{h}u\in W_{h}(\omega^{\ast})$. Hence, $\Pi_{h}H_{A,D}(\omega^{\ast})\subset W_{h}(\omega^{\ast})$. Now we define the continuous linear operators $R_{h}$ in condition I as $R_{h} = \Pi_{h}|_{H_{A,D}(\omega^{\ast})}$. Obviously, by (4.14) we have $\Vert R_{h}u\Vert_{a,\omega^{\ast}}\leq \Vert u\Vert_{a,\omega^{\ast}}$ for any $u\in H_{A,D}(\omega^{\ast})$. In addition, it follows from properties of the orthogonal projection that
\begin{equation}\label{eq:4-15}
\Vert u- R_{h}u\Vert_{a,\omega^{\ast}} = \inf_{v_{h}\in V_{h,D}(\omega^{\ast})}\Vert u-v_{h}\Vert_{a,\omega^{\ast}}.
\end{equation}
Therefore, for any given $u\in H_{A,D}(\omega^{\ast})$, 
\begin{equation}\label{eq:4-16}
\Vert u- R_{h}u\Vert_{a,\omega^{\ast}}\rightarrow 0  \quad {\rm as}\quad  h\rightarrow 0.
\end{equation}
To prove \cref{eq:4-6}, we first observe that by \cref{eq:4-7}, $\{ u_{h}\}$ and $\{ v_{h}\}$ are bounded sequences in $H^{1}_{D}(\omega^{\ast})$. It follows that
\begin{equation}\label{eq:4-17}
\begin{array}{lll}
{\displaystyle |a_{\omega^{\ast}}(u_{h}, v_{h}) - a_{\omega^{\ast}}(u, v)| \leq |a_{\omega^{\ast}}(u_{h}-R_{h}u, v_{h})|  }\\[2mm]
{\displaystyle + |a_{\omega^{\ast}}(R_{h}u-u, v_{h})|  +  |a_{\omega^{\ast}}(u, v_{h}-R_{h}v)| + |a_{\omega^{\ast}}(u, R_{h}v-v)|}\\[2mm]
{\displaystyle \leq C\big( \Vert u_{h}-R_{h}u\Vert_{a,\omega^{\ast}} + \Vert R_{h}u-u\Vert_{a,\omega^{\ast}} + \Vert v_{h}-R_{h}v\Vert_{a,\omega^{\ast}} + \Vert R_{h}v-v\Vert_{a,\omega^{\ast}}\big),}
\end{array}
\end{equation}
which yields \cref{eq:4-6} by applying \cref{eq:4-7,eq:4-16}. Therefore, condition I is verified.

Next, we verify condition II. We have seen that the operators $T_{h}$, $T$ are positive, compact and self-adjoint. Moreover, for any $u\in H^{1}_{D}(\omega^{\ast})$, using the chain rule and the triangle inequality, we get
\begin{equation}\label{eq:4-18}
\Vert \chi u \Vert_{a, \omega^{\ast}} \leq \Big(\int_{\omega^{\ast}}(A\nabla u \cdot \nabla u) \chi^{2}\,d{\bm x}\Big)^{\frac12} +  \Big(\int_{\omega^{\ast}}(A\nabla \chi \cdot \nabla \chi) u^{2}\,d{\bm x}\Big)^{\frac12}.
\end{equation}
Applying \cref{eq:1-1-0} and the Friedrichs inequality, we further have
\begin{equation}\label{eq:4-19}
\Vert \chi u \Vert_{a, \omega^{\ast}}\leq \big(\Vert \chi \Vert_{L^{\infty}(\omega)} + C\beta^{1/2}\Vert \nabla\chi \Vert_{L^{\infty}(\omega)}\big)\Vert u \Vert_{a, \omega^{\ast}}.
\end{equation}
Now taking $v=T_{h}u$ in \cref{eq:4-2} and using \cref{eq:4-19}, we obtain
\begin{equation}\label{eq:4-20}
\begin{array}{lll}
{\displaystyle \Vert T_{h}u \Vert_{a, \omega^{\ast}}\leq \big(\Vert \chi \Vert_{L^{\infty}(\omega)} + C\beta^{1/2}\Vert \nabla\chi \Vert_{L^{\infty}(\omega)}\big)\Vert \chi u \Vert_{a, \omega^{\ast}} }\\[2mm]
{\displaystyle \qquad \quad \qquad \leq \big(\Vert \chi \Vert_{L^{\infty}(\omega)} + C\beta^{1/2}\Vert \nabla\chi \Vert_{L^{\infty}(\omega)}\big)^{2}\Vert u \Vert_{a, \omega^{\ast}}. }
\end{array}
\end{equation}
Therefore, the norms $\Vert T_{h}\Vert=\Vert T_{h}\Vert_{\mathcal{L}(W_{h}(\omega^{\ast}))}$ are bounded by a constant independent of $h$ and thus condition II holds true.

Now we turn to the verification of condition III. First, we observe that
\begin{equation}\label{eq:4-21}
\begin{array}{lll}
{\displaystyle \Vert T_{h}\psi_{h} - R_{h}T\psi \Vert_{a,\omega^{\ast}} \leq \Vert T_{h}(\psi_{h} - R_{h}\psi) \Vert_{a,\omega^{\ast}} + \Vert T_{h}(R_{h}\psi - \psi) \Vert_{a,\omega^{\ast}} }\\[3mm]
{\displaystyle \qquad +\,\Vert R_{h}T\psi - T\psi \Vert_{a,\omega^{\ast}} +\Vert T_h\psi - T\psi \Vert_{a,\omega^{\ast}}.}
\end{array}
\end{equation}
Using \cref{eq:4-8,eq:4-16} and the boundedness of the operators $T_{h}$, we have
\begin{equation}\label{eq:4-22}
\lim_{h\rightarrow0}\big(\Vert T_{h}(\psi_{h} - R_{h}\psi) \Vert_{a,\omega^{\ast}} + \Vert T_{h}(R_{h}\psi - \psi) \Vert_{a,\omega^{\ast}} +\Vert R_{h}T\psi - T\psi \Vert_{a,\omega^{\ast}} \big) = 0.
\end{equation}
It remains to estimate the last term on the right-hand side of \cref{eq:4-21}. Note that problems \cref{eq:4-1,eq:4-2} can be written as the following equivalent mixed formulations: find $Tu\in H^{1}_{D}(\omega^{\ast})$, $p\in H^{1}_{DI}(\omega^{\ast})$ such that
\begin{equation}\label{eq:4-23}
\begin{aligned}
a_{\omega^{\ast}}(Tu,\phi) + a_{\omega^{\ast}}(\phi,p) =&\, a_{\omega}(\chi u, \chi \phi)\quad \forall \phi\in H^{1}_{D}(\omega^{\ast}),\\
a_{\omega^{\ast}}(Tu,\xi) =&\;0\qquad\;\,\quad\qquad \forall \xi\in H^{1}_{DI}(\omega^{\ast}),
\end{aligned}
\end{equation}
and find $T_{h}u\in V_{h,D}(\omega^{\ast})$, $p_{h}\in V_{h,DI}(\omega^{\ast})$ such that
\begin{equation}\label{eq:4-24}
\begin{aligned}
a_{\omega^{\ast}}(T_{h}u,\phi_h) + a_{\omega^{\ast}}(\phi_h,p_h) =&\, a_{\omega}(\chi u, \chi \phi_h)\quad\, \forall \phi_h\in V_{h,D}(\omega^{\ast}),\\
a_{\omega^{\ast}}(T_{h}u,\xi_h) =&\;0\qquad\quad\quad\qquad \forall \xi_{h}\in V_{h,DI}(\omega^{\ast}).
\end{aligned}
\end{equation}
In fact, in \cref{eq:4-23}, the second equation implies that $Tu$ is $A$-harmonic and taking $\phi\in H_{A,D}(\omega^{\ast})$ yields \cref{eq:4-1}; the same holds for \cref{eq:4-24}. Now applying \cref{lem:3-s2-1,lem:3-s2-2} with $\mathcal{F}(\phi) = a_{\omega}(\chi u, \chi \phi)$, we see that there exist unique solutions $(Tu, p)$ and $(T_{h}u, p_{h})$ to \cref{eq:4-23,eq:4-24}, respectively. Moreover, 
\begin{equation}\label{eq:4-25}
\begin{array}{lll}
{\displaystyle \qquad \qquad \Vert Tu - T_{h}u\Vert_{a,\omega^{\ast}}+ \Vert p - p_{h}\Vert_{a,\omega^{\ast}} }\\[3mm]
{\displaystyle \leq C\big(\inf_{v_h\in V_{h,D}(\omega^{\ast})}\Vert Tu-v_h\Vert_{a,\omega^{\ast}} + \inf_{q_h\in V_{h,DI}(\omega^{\ast})}\Vert p-q_h\Vert_{a,\omega^{\ast}}\big). }
\end{array}
\end{equation}
For any fixed $u\in H_{A,D}(\omega^{\ast})$, since $Tu\in H^{1}_{D}(\omega^{\ast})$ and $p\in H^{1}_{DI}(\omega^{\ast})$, it follows from \cref{eq:4-25} that $\Vert Tu-T_{h}u\Vert_{a,\omega^{\ast}}\rightarrow 0$ as $h\rightarrow 0$, which, together with \cref{eq:4-21,eq:4-22}, leads to \cref{eq:4-9}. Hence, condition III is verified.

We proceed to verify condition IV where the discrete Caccioppoli inequality plays an essential role. Let $\{\psi_{h}\}$ be a sequence such that $\psi_{h}\in W_{h}(\omega^{\ast})$ for each $h\in (0,1]$ and $\sup_{h\in (0,1]}\Vert \psi_{h} \Vert_{a,\omega^{\ast}}<\infty$. Since $W_{h}(\omega^{\ast})\subset H^{1}_{D}(\omega^{\ast})$, there exists a subsequence $\{\psi_{h^{\prime}}\}$ and a $\psi\in H^{1}_{D}(\omega^{\ast})$ such that $\{\psi_{h^{\prime}}\}$ converges weakly to $\psi$ in $ H^{1}_{D}(\omega^{\ast})$. Next we show that $\psi\in H_{A,D}(\omega^{\ast})$. Let $\widehat{\Pi}_{h}:H^{1}_{DI}(\omega^{\ast})\rightarrow V_{h,DI}(\omega^{\ast})$ denote the orthogonal projection with respect to the inner product $a_{\omega^{\ast}}(\cdot,\cdot)$. For any fixed $v\in H^{1}_{DI}(\omega^{\ast})$, we note that
\begin{equation}\label{eq:4-26}
a_{\omega^{\ast}}(\psi,v) = a_{\omega^{\ast}}(\psi-\psi_{h^{\prime}},v) + a_{\omega^{\ast}}(\psi_{h^{\prime}}, \widehat{\Pi}_{h^{\prime}}v) + a_{\omega^{\ast}}(\psi_{h^{\prime}}, v-\widehat{\Pi}_{h^{\prime}}v)
\end{equation}
Since $\{\psi_{h^{\prime}}\}$ converges weakly to $\psi$ in $ H^{1}_{D}(\omega^{\ast})$, we have $a_{\omega^{\ast}}(\psi-\psi_{h^{\prime}},v)\rightarrow 0$ as $h^{\prime}\rightarrow 0$. Moreover, since $\psi_{h^{\prime}}\in W_{h^{\prime}}(\omega^{\ast})$ and $\widehat{\Pi}_{h^{\prime}}\in V_{h^{\prime},DI}(\omega^{\ast})$, the second term $a_{\omega^{\ast}}(\psi_{h^{\prime}}, \widehat{\Pi}_{h^{\prime}}v)$ vanishes. Finally, using the boundedness of the sequence $\{\psi_{h^{\prime}}\}$ and properties of the orthogonal projection, we conclude
\begin{equation}\label{eq:4-27}
|a_{\omega^{\ast}}(\psi_{h^{\prime}}, v-\widehat{\Pi}_{h^{\prime}}v)|\leq C\Vert v-\widehat{\Pi}_{h^{\prime}}v\Vert_{a,\omega^{\ast}} \leq C\inf_{v_{h^{\prime}}\in V_{h^{\prime},DI}(\omega^{\ast})} \Vert v-v_{h^{\prime}}\Vert_{a,\omega^{\ast}}\rightarrow 0
\end{equation}
as $h^{\prime}\rightarrow 0$. Now making $h^{\prime}\rightarrow 0$, we see that $a_{\omega^{\ast}}(\psi,v)=0$ for any $v\in H^{1}_{DI}(\omega^{\ast})$ and thus $\psi\in H_{A,D}(\omega^{\ast})$. Define $u=T\psi\in H_{A,D}(\omega^{\ast})$. We now show that $u$ satisfies \cref{eq:4-10}. First, we observe that 
\begin{equation}\label{eq:4-28}
\begin{array}{lll}
{\displaystyle \Vert T_{h^{\prime}}\psi_{h^{\prime}} - R_{h^{\prime}}u \Vert_{a,\omega^{\ast}} = \Vert T_{h^{\prime}}\psi_{h^{\prime}} - R_{h^{\prime}}T\psi\Vert_{a,\omega^{\ast}}  }\\[2mm]
{\displaystyle \leq \Vert T_{h^{\prime}}(\psi_{h^{\prime}}- R_{h^{\prime}}\psi)\Vert_{a,\omega^{\ast}}+ \Vert T_{h^{\prime}}(R_{h^{\prime}}\psi-\psi)\Vert_{a,\omega^{\ast}}}\\[2mm]
{\displaystyle \quad +\,\Vert R_{h^{\prime}}T\psi-T\psi\Vert_{a,\omega^{\ast}}+ \Vert T_{h^{\prime}}\psi - T\psi\Vert_{a,\omega^{\ast}}. }
\end{array}
\end{equation}
It follows from the boundedness of the operators $T_{h^{\prime}}$ and \cref{eq:4-16} that 
\begin{equation}\label{eq:4-29}
\Vert T_{h^{\prime}}(R_{h^{\prime}}\psi-\psi)\Vert_{a,\omega^{\ast}} + \Vert R_{h^{\prime}}T\psi-T\psi\Vert_{a,\omega^{\ast}}\rightarrow 0 \quad {\rm as}\;\,h^{\prime}\rightarrow 0.
\end{equation}
Using \cref{lem:3-s2-1,lem:3-s2-2} and a similar argument as in the verification of condition III, we obtain
\begin{equation}\label{eq:4-30}
\Vert T_{h^{\prime}}\psi - T\psi\Vert_{a,\omega^{\ast}}\rightarrow 0 \quad {\rm as}\;\,h^{\prime}\rightarrow 0.
\end{equation}
We are left with estimating the first term on the right-hand side of \cref{eq:4-28}. Applying \cref{cor:3-1} to $\psi_{h^{\prime}}- R_{h^{\prime}}\psi\in W_{h^{\prime}}(\omega^{\ast})$ and $\chi\in W^{1,\infty}(\omega^{\ast})$ and using \cref{eq:4-20}, we have
\begin{equation}\label{eq:4-31}
\Vert T_{h^{\prime}}(\psi_{h^{\prime}}- R_{h^{\prime}}\psi)\Vert_{a,\omega^{\ast}}\leq C\Vert \chi (\psi_{h^{\prime}}- R_{h^{\prime}}\psi)\Vert_{a,\omega^{\ast}}\leq C\Vert \psi_{h^{\prime}}- R_{h^{\prime}}\psi\Vert_{L^{2}(\omega^{\ast})}.
\end{equation}
Since $\psi_{h^{\prime}}$ weakly converges to $\psi$ in $H^{1}_{D}(\omega^{\ast})$, by the Rellich compactness theorem, we see that $\psi_{h^{\prime}}$ strongly converges to $\psi$ in $L^{2}(\omega^{\ast})$. Therefore,
\begin{equation}\label{eq:4-32}
\Vert \psi_{h^{\prime}}- R_{h^{\prime}}\psi\Vert_{L^{2}(\omega^{\ast})}\leq \Vert \psi_{h^{\prime}}- \psi\Vert_{L^{2}(\omega^{\ast})} + \Vert \psi- R_{h^{\prime}}\psi\Vert_{L^{2}(\omega^{\ast})}\rightarrow 0 \quad {\rm as}\;\,h^{\prime}\rightarrow 0,
\end{equation}
where we have used the fact $ \Vert \psi- R_{h^{\prime}}\psi\Vert_{L^{2}(\omega^{\ast})}\leq C \Vert \psi- R_{h^{\prime}}\psi\Vert_{a,\omega^{\ast}}$. Combining \cref{eq:4-31,eq:4-32} gives 
\begin{equation}\label{eq:4-33}
\Vert T_{h^{\prime}}(\psi_{h^{\prime}}- R_{h^{\prime}}\psi)\Vert_{a,\omega^{\ast}}\rightarrow 0 \quad {\rm as}\;\,h^{\prime}\rightarrow 0.
\end{equation}
\cref{eq:4-10} follows from \cref{eq:4-28,eq:4-29,eq:4-30,eq:4-33} and thus condition IV is verified. 

In the case that $\partial \omega^{\ast}\cap\partial \Omega_{D} = \emptyset$, the verification of conditions I - IV follows the same lines as before except that the proof of \cref{eq:4-25} needs some special care. In fact, \cref{eq:4-1-0,eq:4-2-0} can be written as the following mixed formulations: find $Tu\in H^{1}(\omega^{\ast})/\mathbb{R}$, $p\in H^{1}_{DI}(\omega^{\ast})$ such that
\begin{equation}\label{eq:4-33-3}
\begin{aligned}
a_{\omega^{\ast}}(Tu,\phi) + a_{\omega^{\ast}}(\phi,p) =&\, a_{\omega}(\chi u, \chi \phi)\quad \forall \phi\in H^{1}(\omega^{\ast})/\mathbb{R},\\
a_{\omega^{\ast}}(Tu,\xi) =&\;0\qquad\;\quad\qquad \forall \xi\in H^{1}_{DI}(\omega^{\ast}),
\end{aligned}
\end{equation}
and find $T_{h}u\in V_{h}(\omega^{\ast})/\mathbb{R}$, $p_{h}\in V_{h,DI}(\omega^{\ast})$ such that
\begin{equation}\label{eq:4-33-4}
\begin{aligned}
a_{\omega^{\ast}}(T_{h}u,\phi_h) + a_{\omega^{\ast}}(\phi_h,p_h) =&\, a_{\omega}(\chi u, \chi \phi_h)\quad\, \forall \phi_h\in V_{h}(\omega^{\ast})/\mathbb{R},\\
a_{\omega^{\ast}}(T_{h}u,\xi_h) =&\;0\qquad\quad\quad\qquad \forall \xi_{h}\in V_{h,DI}(\omega^{\ast}).
\end{aligned}
\end{equation}
Note that $a_{\omega}(\chi u, \chi c) = 0$ for any $c\in \mathbb{R}$ since $u\in \interior{H}_{A,D}(\omega^{\ast})$. By Remark~\ref{rem:3-1}, we get a similar approximation result as \cref{eq:4-25}.

Since we have verified conditions I - IV for both cases, we see that $\mu_{h,k}\rightarrow\mu_{k}$ for each $k\in\mathbb{N}$ as $h\rightarrow 0$ and thus complete the proof of this theorem.
\end{proof}

In some special cases, we can derive the rate of convergence with respect to $h$ for the eigenvalues.
\begin{theorem}\label{thm:4-2}
For $k\in \mathbb{N}$, let $\mu_{k}$ and $\mu_{h,k}$ be the eigenvalues of problems \cref{eq:4-3,eq:4-4}, respectively. Assume that the domain $\omega^{\ast}$ is convex with $\partial \omega^{\ast}\cap\partial \Omega=\emptyset$ and the coefficient $A$ is Lipschitz continuous in $\omega^{\ast}$. In addition, let $\chi\in W^{2,\infty}(\omega)\cap H_{0}^{1}(\omega)$. Then for sufficiently small $h$, 
\begin{equation}\label{eq:4-34}
|\mu_{k} - \mu_{h,k}|\leq Ch,
\end{equation}
where the constant $C$ is independent of $h$ and $k$.
\end{theorem}
\begin{proof}
By \cref{thm:4-1}, it suffices to estimate the term on the right-hand side of \cref{eq:4-12}. Given $k\in \mathbb{N}$, let $u\in N(\mu_{k},T)$ with $\Vert u\Vert_{a,\omega^{\ast}} =1$, where $N(\mu_{k},T)$ is defined in \cref{eq:4-13}. By definition, we see that $a_{\omega^{\ast}}(u,v)=0$ for each $v\in H_{0}^{1}(\omega^{\ast})$, which, together with the Lipschitz continuity of the coefficient $A$, yields that $u\in H^{2}_{loc}(\omega^{\ast})$ with the estimate \cite[Theorem 8.8]{gilbarg2015elliptic} 
\begin{equation}\label{eq:4-35}
\Vert u\Vert_{H^{2}({\omega}^{\prime})}\leq C\Vert u\Vert_{H^{1}(\omega^{\ast})}\leq C
\end{equation}
for any ${\omega}^{\prime}\subset\subset \omega^{\ast}$. In view of the assumptions on $\omega$ and $\chi$, we see that $\chi u\in H^{2}(\omega)$. Next we show that $Tu \in H^{2}(\omega^{\ast})$. By definition, $Tu\in \interior{H}_{A,D}(\omega^{\ast})$ satisfies
\begin{equation}\label{eq:4-36}
a_{\omega^{\ast}}(Tu, v) = a_{\omega}(\chi u, \chi v),\quad \forall v\in \interior{H}_{A,D}(\omega^{\ast}),
\end{equation}
where $\interior{H}_{A,D}(\omega^{\ast})$ is defined in \cref{eq:4-33-0}. Since $A$ is assumed to be Lipschitz continuous and $\chi u\in H^{2}(\omega)$, we have $\nabla\cdot (A\nabla (\chi u))\in L^{2}(\omega)$. Applying integration by parts to the right-hand side of \cref{eq:4-36} gives that
\begin{equation}\label{eq:4-36-0}
a_{\omega^{\ast}}(Tu, v)=-\big(\chi \nabla\cdot (A\nabla (\chi u)),\, v\big)_{L^{2}(\omega^{\ast})}, \quad \forall v\in \interior{H}_{A,D}(\omega^{\ast}).
\end{equation}
Similar to \cref{eq:4-33-3}, \cref{eq:4-36-0} can be written as the following mixed formulation: find $Tu\in H^{1}(\omega^{\ast})/\mathbb{R}$, $p\in H^{1}_{0}(\omega^{\ast})$ such that
\begin{equation}\label{eq:4-37}
\begin{aligned}
a_{\omega^{\ast}}(Tu,\phi) + a_{\omega^{\ast}}(\phi,p) =&\, -\big(\chi \nabla\cdot (A\nabla (\chi u)),\, \phi\big)_{L^{2}(\omega^{\ast})}\quad \forall \phi\in H^{1}(\omega^{\ast})/\mathbb{R},\\
a_{\omega^{\ast}}(Tu,\xi) =&\;0\qquad \qquad\qquad\qquad \qquad \,\quad\qquad \forall \xi\in H^{1}_{0}(\omega^{\ast}).
\end{aligned}
\end{equation}
Note that $H^{1}_{DI}(\omega^{\ast}) = H_{0}^{1}(\omega^{\ast})$ since we have assumed that $\partial \omega^{\ast}\cap\partial \Omega = \emptyset$. Taking $\phi\in H_{0}^{1}(\omega^{\ast})$ (up to an additive constant), we see that $p\in H^{1}_{0}(\omega^{\ast})$ satisfies
\begin{equation}\label{eq:4-39}
 a_{\omega^{\ast}}(\phi,p) = -\big(\chi \nabla\cdot (A\nabla (\chi u)),\, \phi\big)_{L^{2}(\omega^{\ast})},\quad \forall \phi\in H_{0}^{1}(\omega^{\ast}).
\end{equation}
Having assumed the convexity of the domain $\omega^{\ast}$ and the Lipschitz continuity of the coefficient $A$, by applying Theorem 3.2.1.2 of \cite{grisvard2011elliptic} and using the fact that the right-hand term $-\chi \nabla\cdot (A\nabla (\chi u))\in L^{2}(\omega^{\ast})$, we obtain that $p\in H^{2}(\omega^{\ast})$ and 
\begin{equation}\label{eq:4-40}
\Vert p\Vert_{H^{2}(\omega^{\ast})}\leq C\Vert \chi \nabla\cdot (A\nabla (\chi u))\Vert_{L^{2}(\omega^{\ast})}\leq C\Vert u\Vert_{H^{2}(\omega)}\leq C.
\end{equation}
Therefore, we have $\nabla\cdot(A \nabla p)\in L^{2}(\omega^{\ast})$ and $A\nabla p\cdot{\bm n}\in H^{1/2}(\partial \omega^{\ast})$. Using integration by parts again gives 
\begin{equation}\label{eq:4-41}
a_{\omega^{\ast}}(\phi,p) = -\big(\nabla\cdot(A \nabla p),\, \phi\big)_{L^{2}(\omega^{\ast}) }+ \langle A\nabla p\cdot{\bm n},\, \phi\rangle_{L^{2}(\partial \omega^{\ast})},\quad \forall \phi\in H^{1}(\omega^{\ast}).
\end{equation}
Combining \cref{eq:4-37,eq:4-41}, we see that $Tu\in H^{1}(\omega^{\ast})/\mathbb{R}$ satisfies
\begin{equation}\label{eq:4-42}
a_{\omega^{\ast}}(Tu,\phi) = \big(\nabla\cdot(A \nabla p) -\chi \nabla\cdot (A\nabla (\chi u))\big)_{L^{2}(\omega^{\ast})} -  \langle A\nabla p\cdot{\bm n},\, \phi\rangle_{L^{2}(\partial \omega^{\ast})} 
\end{equation}
for each $\phi\in H^{1}(\omega^{\ast})/\mathbb{R}$. \Cref{eq:4-42} is the weak formulation of the  problem
\begin{equation}\label{eq:4-43}
\left\{
\begin{array}{lll}
{\displaystyle -{\rm div}(A\nabla \,Tu) = f^{\ast},\quad {\rm in}\;\, \omega^{\ast} }\\[2mm]
{\displaystyle {\bm n} \cdot A\nabla \,Tu=g^{\ast}, \quad \quad \;\;{\rm on}\;\,\partial \omega^{\ast},}
\end{array}
\right.
\end{equation}
where $f^{\ast}=\nabla\cdot(A \nabla p)-\chi \nabla\cdot (A\nabla (\chi u))\in L^{2}(\omega^{\ast})$ and $g^{\ast} = -A\nabla p\cdot{\bm n}\in H^{1/2}(\partial \omega^{\ast})$. Using the assumptions on $\omega^{\ast}$ and $A$ again and applying Theorem 3.2.1.3 of \cite{grisvard2011elliptic} (subtracting a function $\psi \in H^{2}(\omega^{\ast})$ from $Tu$ such that ${\bm n} \cdot A\nabla (Tu-\psi) =0$ on $\partial \omega^{\ast}$ if necessary), we get that $Tu\in H^{2}(\omega^{\ast})$ with the estimate
\begin{equation}\label{eq:4-44}
\Vert Tu \Vert_{H^{2}(\omega^{\ast})}\leq C\big(\Vert f^{\ast}\Vert_{L^{2}(\omega^{\ast})} + \Vert g^{\ast}\Vert_{H^{1/2}(\partial \omega^{\ast})}\big)\leq C. 
\end{equation} 
Now we can estimate $\Vert T_{h}R_{h} u - R_{h}Tu\Vert_{a,\omega^{\ast}}$ for $u\in  N(\mu_{k},T)$ with $\Vert u \Vert_{a,\omega^{\ast}} =1$, which can be bounded by
\begin{equation}\label{eq:4-45}
\begin{array}{lll}
{\displaystyle \Vert T_{h}R_{h} u - R_{h}Tu\Vert_{a,\omega^{\ast}}\leq \Vert R_{h}Tu-Tu\Vert_{a,\omega^{\ast}}}\\[2mm]
{\displaystyle \quad + \,\Vert T_{h}u - Tu\Vert_{a,\omega^{\ast}}+\Vert T_{h}(R_{h}u-u)\Vert_{a,\omega^{\ast}} . }
\end{array}
\end{equation}
In what follows, we derive the upper bound for the right-hand side of \cref{eq:4-45} term by term. First, using \cref{eq:4-15}, \cref{eq:4-44}, and the finite element interpolation error estimates, we have
\begin{equation}\label{eq:4-46}
\begin{array}{lll}
{\displaystyle \Vert R_{h}Tu-Tu\Vert_{a,\omega^{\ast}}\leq C\inf_{v_h\in V_{h}(\omega^{\ast})}\Vert Tu-v_h\Vert_{a,\omega^{\ast}} }\\[3mm]
{\displaystyle \leq C\Vert Tu-I_{h}(Tu)\Vert_{a,\omega^{\ast}}\leq Ch\Vert Tu \Vert_{H^{2}(\omega^{\ast})}\leq Ch,}
\end{array}
\end{equation}
where $I_{h}$ denotes the standard Lagrange interpolant. Second, it follows from \cref{eq:4-25,eq:4-40,eq:4-44} that
\begin{equation}\label{eq:4-47}
\begin{array}{lll}
{\displaystyle \Vert Tu - T_{h}u\Vert_{a,\omega^{\ast}}\leq C\big(\inf_{v_h\in V_{h}(\omega^{\ast})}\Vert Tu-v_h\Vert_{a,\omega^{\ast}} + \inf_{q_h\in V_{h,DI}(\omega^{\ast})}\Vert p-q_h\Vert_{a,\omega^{\ast}}\big). }\\[3mm]
{\displaystyle\qquad \qquad \qquad \quad \,\leq Ch\big(\Vert Tu \Vert_{H^{2}(\omega^{\ast})} + \Vert p \Vert_{H^{2}(\omega^{\ast})}\big)\leq Ch.}
\end{array}
\end{equation}
Finally, by \cref{eq:4-20} and the interior estimates for Ritz--Galerkin methods \cite{demlow2011local}, it follows that if $h$ is sufficiently small, 
\begin{equation}\label{eq:4-48-0}
\begin{array}{lll}
{\displaystyle \Vert T_{h}(R_{h}u-u) \Vert_{a,\omega^{\ast}} \leq C\Vert R_{h}u-u \Vert_{a,\omega}}\\[2mm]
{\displaystyle \leq C\Big(\inf_{v_h\in V_{h}(\omega^{\ast})}\Vert u-v_h\Vert_{H^{1}(\omega^{\prime})}  + \Vert R_{h}u-u \Vert_{L^{2}(\omega^{\prime})}\Big),}
\end{array}
\end{equation}
where $\omega\subset\subset \omega^{\prime}\subset\subset \omega^{\ast}$. With \cref{eq:4-35} and the interpolation error estimates, we have
\begin{equation}\label{eq:4-48-1}
\inf_{v_h\in V_{h}(\omega^{\ast})}\Vert u-v_h\Vert_{H^{1}(\omega^{\prime})} \leq Ch \Vert u\Vert_{H^{2}(\omega^{\prime})}\leq Ch.
\end{equation}
Moreover, the assumptions on the domain $\omega^{\ast}$ and the coefficient $A$ allow us to use the Aubin--Nitsche duality argument to obtain
\begin{equation}\label{eq:4-48-2}
\Vert R_{h}u-u \Vert_{L^{2}(\omega^{\prime})}\leq \Vert R_{h}u-u \Vert_{L^{2}(\omega^{\ast})} \leq Ch \Vert R_{h}u-u \Vert_{a,\omega^{\ast}}\leq Ch.
\end{equation}
Inserting \cref{eq:4-48-1,eq:4-48-2} into \cref{eq:4-48-0}, we come to
\begin{equation}\label{eq:4-48-3}
\Vert T_{h}(R_{h}u-u) \Vert_{a,\omega^{\ast}} \leq Ch.
\end{equation}
Combining \cref{eq:4-45,eq:4-46,eq:4-47,eq:4-48-3} leads to 
\begin{equation}\label{eq:4-48-4}
\sup_{\substack{u\in N(\mu_k, T),\\ \Vert u\Vert_{a,\omega^{\ast}}=1}}\Vert T_{h}R_{h} u - R_{h}Tu\Vert_{a,\omega^{\ast}}\leq Ch.
\end{equation}
Now \cref{eq:4-34} follows immediately from \cref{eq:4-12} and \cref{eq:4-48-4}.
\end{proof}
\begin{rem}
Note that the convergence $\mu_{h,k}\rightarrow\mu_{k}$ in \cref{thm:4-2} is uniform with respect to $k$ since the constant $C$ in \cref{eq:4-34} is independent of $k$. Compared to the usual quadratic convergence rate for eigenvalues in linear FE  approximations of PDE eigenvalue problems, the convergence result in \cref{thm:4-2} seems suboptimal. Note that our eigenvalue problem is defined on $H^{1}(\omega^{\ast})$ instead of $L^{2}(\omega^{\ast})$. The key to deriving the optimal convergence rate lies in finding a suitable Caccioppoli-type inequality (see \cref{lem:3-s2-3,lem:3-s2-4}) for a function $v-v_{h}$ with $v\in H_{A,D}(\omega^{\ast})$ and $v_{h}\in W_{h}(\omega^{\ast})$, which is unknown to us. We will investigate this problem in our future work.
\end{rem}

\subsection{Nearly exponential decay of the $n$-width} By \cref{thm:3-2}, the local approximation errors are bounded by the $n$-widths. In \cite{ma2021novel}, we proved that the decay rate of the $n$-width associated with the continuous problem is nearly exponential with respect to $n$, namely \cref{thm:2-2}. In this subsection, we derive a similar upper bound for the $n$-width in the discrete setting. For simplicity, we assume that $\omega$ and $\omega^{\ast}$ are (truncated) concentric cubes with side lengths $H$ and $H^{\ast}$, respectively. Under this assumption, we have
\begin{theorem}\label{thm:4-2-1}
For $\epsilon\in (0,\frac{1}{d+1})$, there exists an $n_{\epsilon}>0$, such that for any $n>n_{\epsilon}$, if $h$ is sufficiently small, then
\begin{equation}\label{eq:4-57}
d_{n}(\omega,\omega^{\ast}) \leq e^{2}(1 + C_{1})e^{-n^{(\frac{1}{d+1}-\epsilon)}} e^{-n^{(\frac{1}{d+1}-\epsilon)}R(\rho)},
\end{equation}
where $C_{1}$ is the constant given in \cref{eq:1-3}, $R(s) = 1+{s\log(s)}/{(1-s)}$, and $\rho = H/H^{\ast}$.
\end{theorem}

The key of the proof is to explicitly construct an $n$-dimensional subspace $\hat{Q}(n)$ of $H^{1}_{DI}(\omega)$ such that the approximation error decays nearly exponentially. To this end, we first consider the following eigenvalue problem
\begin{equation}\label{eq:4-58}
a_{\omega^{\ast}}(v_{k,h}, \varphi) = \lambda_{k,h} (v_{k,h},\,\varphi)_{L^{2}(\omega^{\ast})},\quad \forall \varphi \in V_{h,D}(\omega^{\ast}),\quad k=1,2\cdots.
\end{equation}
Let $\Psi^{m}_{h}(\omega^{\ast})$ denote the space spanned by the first $m$ eigenfunctions of \cref{eq:4-58}. Denoting by $\mathcal{P}^{A}_{h}:V_{h,D}(\omega^{\ast})\rightarrow W_{h}(\omega^{\ast})$ the orthogonal projection with respect to the energy inner product $a_{\omega^{\ast}}(\cdot,\,\cdot)$, we further define $\Theta_{h}^{m}(\omega^{\ast}) = \mathcal{P}^{A}_{h}\Psi^{m}_{h}(\omega^{\ast})$. We have the following approximation result for the space $\Theta_{h}^{m}(\omega^{\ast})$.
\begin{lemma}\label{lem:4-2-1}
For any $u\in W_{h}(\omega^{\ast})$, there exists a $v_{u}\in \Theta_{h}^{m}(\omega^{\ast})$ such that
\begin{equation}\label{eq:4-59}
\Vert u - v_{u}\Vert_{L^{2}(\omega^{\ast})} = \inf_{v\in \Theta_{h}^{m}(\omega^{\ast})} \Vert u - v\Vert_{L^{2}(\omega^{\ast})} \leq C(m)H^{\ast}\frac{\gamma_{d}^{1/d}}{\sqrt{4\pi}}\alpha^{-1/2}\Vert u \Vert_{a,\omega^{\ast}},
\end{equation}
where $H^{\ast}$ is the side length of the cube $\omega^{\ast}$, $\gamma_{d}$ is the volume of the unit ball in $\mathbb{R}^{d}$, and $C(m) = m^{-1/d}(1+o(1))$. 
\end{lemma}
\Cref{lem:4-2-1} can be proved by following the same lines as in the proof of its continuous counterpart in \cite{babuska2011optimal,ma2021novel} and using the min-max principle \cite[Chapter VI]{Courant1989} which shows that all eigenvalues are approximated from above by a conforming approximation. Combining \cref{lem:4-2-1} and \cref{lem:3-s2-4} (the discrete Caccioppoli inequality) gives the approximation error in the energy norm as follows.
\begin{lemma}\label{lem:4-2-2}
Let $\delta^{\ast}=H^{\ast}-H$. For any $u\in W_{h}(\omega^{\ast})$, there exists a $v_{u}\in \Theta_{h}^{m}(\omega^{\ast})$ such that if $h\leq \frac{1}{8}\delta^{\ast}$,
\begin{equation}\label{eq:4-60}
\Vert u - v_{u}\Vert_{a,\,\omega} \leq Cm^{-1/d}\frac{H^{\ast}}{\delta^{\ast}}\Vert u \Vert_{a,\omega^{\ast}},
\end{equation}
where the constant $C$ only depends on $\alpha$, $\beta$, $d$, and the shape-regularity of the mesh.
\end{lemma}

In order to derive a low-dimensional subspace in $H^{1}_{DI}(\omega)$ with the approximation error decaying nearly exponentially, we iterate the approximation result in \cref{lem:4-2-2} on some nested concentric cubes between $\omega$ and $\omega^{\ast}$. For an integer $N\geq 1$, let $\{\omega^{j}\}_{j=1}^{N+1}$ denote the nested family of (truncated) concentric cubes with side length $H^{\ast}-\delta^{\ast} (j-1)/N$ for which $\omega=\omega^{N+1}\subset \omega^{N}\subset\cdots\subset\omega^{1} = \omega^{\ast}$, where $\delta^{\ast} = H^{\ast}-H$. We assume that $h$ is sufficiently small such that $4h\leq {dist}\,(\omega^{j+1},\omega^{j}/\partial \Omega) = \delta^{\ast}/(2N)$. 
Let $n=Nm$ and set
\begin{equation}\label{eq:4-61}
\mathcal{T}_{h}(n,\omega,\omega^{\ast}) = \Theta_{h}^{m}(\omega^{1})+ \cdots+\Theta_{h}^{m}(\omega^{N}),
\end{equation}
we have
\begin{lemma}\label{lem:4-2-3}
Let $u\in W_{h}(\omega^{\ast})$ and assume that $N\geq (\delta^{\ast}/H)^{2}$. Then there exists a $z_{u}\in \mathcal{T}_{h}(n,\omega,\omega^{\ast})$ such that
\begin{equation}\label{eq:4-62}
\Vert \chi(u -z_{u}) \Vert_{a,\omega} \leq \big(1+\frac{C_{1}}{2\sqrt{2N}}\big)\prod_{k=1}^{N-1} \big(1-\frac{k\delta^{\ast}}{NH^{\ast}}\big) \xi^{N} \Vert u \Vert_{a,\omega^{\ast}},
\end{equation}
where $C_{1}$ is the (positive) constant introduced in \cref{eq:1-3} and $\xi$ is given by
\begin{equation}\label{eq:4-63}
\xi = \xi(N,m)= CNm^{-1/d}H^{\ast}/\delta^{\ast}.
\end{equation}
\end{lemma}
\begin{proof}
Applying \cref{lem:4-2-2} to $\omega^{1}$ and $\omega^{2}$ with side lengths $H^{\ast}$ and $H^{\ast}-\delta^{\ast}/N$, we see that there exists a $v_{u}^{1}\in \Theta_{h}^{m}(\omega^{1})$ such that
\begin{equation}\label{eq:4-64}
\Vert u - v_{u}^{1} \Vert_{a,\omega^{2}} \leq \xi \Vert u \Vert_{a,\omega^{1}} = \xi \Vert u \Vert_{a,\omega^{\ast}}.
\end{equation}
Since $(u-v_{u}^{1})|_{\omega^{2}}\in W_{h}(\omega^{2})$, it follows from \cref{lem:4-2-2} again that we can find a function $v_{u}^{2}\in \Theta_{h}^{m}(\omega^{2})$ such that
\begin{equation}\label{eq:4-65}
\begin{array}{lll}
{\displaystyle  \Vert u-v_{u}^{1} - v_{u}^{2} \Vert_{a,\omega^{3}} \leq \big(1-\frac{\delta^{\ast}}{NH^{\ast}}\big)\xi \Vert u-v_{u}^{1} \Vert_{a,\omega^{2}} \leq \big(1-\frac{\delta^{\ast}}{NH^{\ast}}\big)\xi^{2} \Vert u \Vert_{a,\omega^{\ast}}.}
\end{array}
\end{equation}
Applying the same argument above until $k=N-1$, we see that there exist $v_{u}^{k}\in \Theta_{h}^{m}(\omega^{k})$, $k=1,\cdots,N-1$, such that
\begin{equation}\label{eq:4-66}
\Big \Vert u - \sum_{k=1}^{N-1}v_{u}^{k} \Big\Vert_{a,\omega^{N}} \leq \prod_{k=1}^{N-2} \big(1-\frac{k\delta^{\ast}}{NH^{\ast}}\big)\xi^{N-1} \Vert u \Vert_{a,\omega^{\ast}}.
\end{equation}
Finally, combining \cref{lem:4-2-1}, \cref{cor:3-1}, and \cref{eq:4-66}, we can find a $v_{u}^{N}\in \Theta_{h}^{m}(\omega^{N})$ such that
\begin{equation}\label{eq:4-67}
\begin{array}{lll}
{\displaystyle  \Big\Vert \chi\big(u - \sum_{k=1}^{N-1}v_{u}^{k} -  v_{u}^{N}\big) \Big\Vert_{a,\omega}\leq \big(CN/\delta^{\ast} + \beta^{1/2}\Vert \nabla\chi\Vert_{L^{\infty}(\omega)}\big)\Big\Vert u - \sum_{k=1}^{N}v_{u}^{k}\Big\Vert_{L^{2}(\omega^{N})}}\\[4mm]
{\displaystyle \leq C\big(N/\delta^{\ast} + \Vert \nabla\chi\Vert_{L^{\infty}(\omega)}\big)m^{-1/d}\big(H^{\ast}-\delta^{\ast}(N-1)/N\big)\Big\Vert u - \sum_{k=1}^{N-1}v_{u}^{k} \Big\Vert_{a,\omega^{N}} }\\[3mm]
{\displaystyle \leq \big(1+ \frac{C_{1}\delta^{\ast}}{2\sqrt{2}NH}\big)\prod_{k=1}^{N-1} \big(1-\frac{k\delta^{\ast}}{NH^{\ast}}\big)\xi^{N} \Vert u \Vert_{a,\omega^{\ast}}}\\[3mm]
{\displaystyle \leq \big(1+\frac{C_{1}}{2\sqrt{2N}}\big)\prod_{k=1}^{N-1} \big(1-\frac{k\delta^{\ast}}{NH^{\ast}}\big) \xi^{N} \Vert u \Vert_{a,\omega^{\ast}},} 
\end{array}
\end{equation}
where we have used the inequality $\Vert \nabla \chi\Vert_{L^{\infty}(\omega)} \leq {C_{1}}/{diam\,(\omega)}$ and the assumption that $N\geq (\delta^{\ast}/H)^{2}$. Therefore, $z_{u} = \sum_{k=1}^{N}v_{u}^{k} \in \mathcal{T}_{h}(n,\omega,\omega^{\ast})$ satisfies \cref{eq:4-62}.
\end{proof}

{\it Proof of \cref{thm:4-2-1}}. We define the desired approximation space as 
\begin{equation}
\hat{Q}(n) = \chi \mathcal{T}_{h}(n,\omega,\omega^{\ast})\subset H^{1}_{DI}(\omega).
\end{equation}
It follows from the definition of the $n$-width and \cref{lem:4-2-3} (increasing the constant in $\xi$ if necessary) that 
\begin{equation}\label{eq:4-68}
\begin{array}{lll}
{\displaystyle d_{n}(\omega, \omega^{\ast}) \leq \sup_{u\in W_{h}(\omega^{\ast})} \inf_{v\in \hat{Q}(n)}\frac {\Vert \chi u-v\Vert_{a,\omega}}{\Vert u \Vert_{a,\omega^{\ast}}}\,\leq \,\frac{1+C_{1}}{2\sqrt{2N}}\prod_{k=1}^{N-1} \big(1-\frac{k\delta^{\ast}}{NH^{\ast}}\big) \xi^{N}.}
\end{array}
\end{equation}

Using an almost identical argument as in the proof of Theorem 3.5 of \cite{ma2021novel}, it can be shown that for any $\epsilon\in (0,\frac{1}{d+1})$, if 
\begin{equation}\label{eq:4-69}
n\geq n_{\epsilon} = \max \big\{ \big(CH^{\ast}/\delta^{\ast}\big)^{ \frac{d}{\epsilon(1+d)}},  \,\,\big(\delta^{\ast}/H\big)^{\frac{2(d+1)}{1-\epsilon(d+1)}}\big\}
\end{equation}
and if $h$ is sufficiently small, then 
\begin{equation}\label{eq:4-70}
\begin{array}{lll}
{\displaystyle \frac{1+C_{1}}{2\sqrt{2N}}\prod_{k=1}^{N-1} \big(1-\frac{k\delta^{\ast}}{NH^{\ast}}\big) \xi^{N}  \leq e^{2}(1+ C_{1}) e^{-n^{(\frac{1}{d+1}-\epsilon)}} e^{-n^{(\frac{1}{d+1}-\epsilon)}R(\rho)},}
\end{array}
\end{equation}
where $R(s) = 1+{s\log(s)}/{(1-s)}$ and $\rho = H/H^{\ast}$. Combining \cref{eq:4-68,eq:4-70} yields \cref{eq:4-57}. \qquad \qquad $\square$

\section{Eigensolver based on mixed formulation}\label{sec-5}
In this section, we present an efficient and accurate way of solving the local eigenvalue problems \cref{eq:3-5}. As is done in the preceding section, we introduce a Lagrange multiplier to relax the $A$-harmonic constraint and rewrite \cref{eq:3-5} in an equivalent mixed formulation: Find $\lambda_{h}\in\mathbb{R}$, $\phi_{h}\in V_{h,D}(\omega^{\ast})$, and $p_{h}\in V_{h,DI}(\omega^{\ast})$ such that
\begin{equation}\label{eq:5-1}
\begin{aligned}
a_{\omega^{\ast}}(\phi_h,v_{h}) + a_{\omega^{\ast}}(v_h,p_h) =&\, \lambda_{h}a_{\omega}(\chi \phi_{h}, \chi v_h)\quad\, \forall v_h\in V_{h,D}(\omega^{\ast}),\\
a_{\omega^{\ast}}(\phi_{h},\xi_h) =&\;0\qquad\qquad\qquad\quad \;\;\forall \xi_{h}\in V_{h,DI}(\omega^{\ast}).
\end{aligned}
\end{equation}
We partition the degrees of freedom (DOFs) associated with $V_{h,D}(\omega^{\ast})$ into a set $\mathcal{B}_{1}$ of DOFs associated with $V_{h,DI}(\omega^{\ast})$ and the remainder $\mathcal{B}_{2}$, which contains those DOFs that are "active" on the interior boundary of $\omega^{\ast}$, i.e.,
\begin{equation}\label{eq:5-2}
\begin{array}{lll}
{\displaystyle \mathcal{B}_{1}:=\{k=1,\cdots,\mathcal{N}\;:\; \varphi_{k}|_{\omega^{\ast}}\in V_{h,DI}(\omega^{\ast})\},}\\[2mm]
{\displaystyle \mathcal{B}_{2}:=\{k=1,\cdots,\mathcal{N}\;:\;\varphi_{k}|_{\omega^{\ast}}\in V_{h,D}(\omega^{\ast}),\; \varphi_{k}|_{\omega^{\ast}}\notin V_{h,DI}(\omega^{\ast}) \},}
\end{array}
\end{equation}
where $\{\varphi_{j}\}_{j=1}^{\mathcal{N}}$ is the basis of $V_{h}$ introduced at the beginning of \cref{sec-3}. In addition, we define $n_{1} =\#\mathcal{B}_{1}$ and $n_{2}=\# \mathcal{B}_{2}$. With these notations, \cref{eq:5-1} can be formulated as a matrix eigenvalue problem: Find $\lambda\in\mathbb{R}$, ${\bm \phi} = ({\bm \phi}_{1}, {\bm \phi}_{2})\in\mathbb{R}^{n_{1}+n_{2}}$, and ${\bm p}\in \mathbb{R}^{n_1}$ such that
\begin{equation}\label{eq:5-3}
{\left( \begin{array}{ccc}
{\bf A}_{11} & {\bf A}_{12} & {\bf A}_{11}\\
{\bf A}_{21} & {\bf A}_{22} & {\bf A}_{21}\\
{\bf A}_{11} & {\bf A}_{12} & {\bf 0}
\end{array} 
\right )}
{\left( \begin{array}{c}
{\bm \phi}_{1} \\
{\bm \phi}_{2} \\
{\bm p}
\end{array} 
\right )} = \lambda
{\left( \begin{array}{ccc}
{\bf B}_{11} & {\bf 0} & {\bf 0}\\
{\bf 0} & {\bf 0} & {\bf 0}\\
{\bf 0} & {\bf 0} & {\bf 0}
\end{array} 
\right )}
{\left( \begin{array}{c}
{\bm \phi}_{1} \\
{\bm \phi}_{2} \\
{\bm p}
\end{array} 
\right )},
\end{equation}
where ${\bf A}_{ij}=\big(a_{\omega^{\ast}}(\varphi_{k}, \varphi_{l})\big)_{k\in \mathcal{B}_{i}, l\in \mathcal{B}_{j}}$ and ${\bf B}_{11}=\big(a_{\omega}(\chi\varphi_{k}, \chi\varphi_{l})\big)_{k\in \mathcal{B}_{1}, l\in \mathcal{B}_{1}}$. Here the blocks $\mathbf{B}_{12}$, $\mathbf{B}_{21}$, and $\mathbf{B}_{22}$ vanish due to the fact that the interior boundaries of $\omega$ and $\omega^{\ast}$ are disjoint such that ${\rm supp}\,(\varphi_{k})\cap \omega = \emptyset$ for all $k\in \mathcal{B}_{2}$. By block-elimination, it follows that 
\begin{equation}\label{eq:5-4}
\mathbf{A}_{11}{\bm p} = \lambda \mathbf{B}_{11}{\bm \phi}_{1} \quad \Longleftrightarrow \quad \mathbf{B}_{11}{\bm \phi}_{1} = \lambda^{-1} \mathbf{A}_{11}{\bm p},
\end{equation}
where $ {\bm \phi}_{1}$ can be computed from ${\bm p}$ by 
\begin{equation}\label{eq:5-5}
{\bf A}{\left( \begin{array}{c}
{\bm \phi}_{1} \\
{\bm \phi}_{2} 
\end{array} 
\right )} = 
{\left( \begin{array}{c}
{\bf 0} \\
-\mathbf{A}_{21}{\bm p}
\end{array} 
\right )}, 
\quad {\rm with}\;\,{\bf A}={\left( \begin{array}{cc}
{\bf A}_{11} & {\bf A}_{12} \\
{\bf A}_{21} & {\bf A}_{22} 
\end{array} 
\right )}.
\end{equation}

\Cref{eq:5-4} can be viewed as a generalized eigenvalue problem concerning ${\bm p}$ where a function involving solving \cref{eq:5-5} is specified for the matrix-vector product on the left-hand side. Note that the block $\mathbf{A}_{11}$ is the matrix version of the bilinear form $a_{\omega^{\ast}}(\cdot,\cdot): V_{h,DI}(\omega^{\ast})\times V_{h,DI}(\omega^{\ast}) \rightarrow \mathbb{R}$ and thus it is positive definite. If $\partial \omega^{\ast} \cap \partial \Omega_{D}\neq \emptyset$, then the matrix ${\bf A}$ is positive definite and we can preform an $LDL^{T}$ decomposition ${\bf A}={\bf L}{\bf D}{\bf L}^{T}$, where ${\bf L}$ is a lower triangular matrix and $\mathbf{D}$ is a diagonal matrix. In this way, ${\bm \phi}_{1}$ can be computed by solving an upper and lower triangular system:
\begin{equation}\label{eq:5-8}
{\left( \begin{array}{c}
{\bm \phi}_{1} \\
{\bm \phi}_{2} 
\end{array} 
\right )} = \mathbf{L}^{-T}({\bf L}{\bf D})^{-1}
{\left( \begin{array}{c}
{\bf 0} \\
-\mathbf{A}_{21}{\bm p}
\end{array} 
\right )}.
\end{equation}

If $\partial \omega^{\ast}\cap \partial \Omega_{D} = \emptyset$, the "Neumann" matrix ${\bf A}$ is positive semi-definite and thus not invertible. By taking $v_{h}=1$ in \cref{eq:5-1}, we observe that all the eigenfunctions corresponding to positive eigenvalues satisfy $a_{\omega}(\chi \phi_{h}, \chi) = 0$. To make the linear system \cref{eq:5-5} solvable, we impose this constraint on the solution and introduce another variable $c\in \mathbb{R}$ to modify \cref{eq:5-5}, such that 
\begin{equation}\label{eq:5-12}
{\bf \widehat{A}}{\left( \begin{array}{c}
{\bm \phi}_{1} \\
{\bm \phi}_{2} \\
c
\end{array} 
\right )} = 
{\left( \begin{array}{c}
{\bf 0} \\
-\mathbf{A}_{21}{\bm p}\\
0
\end{array} 
\right )}, 
\quad {\rm with}\;\,\widehat{\bf A}={\left( \begin{array}{ccc}
{\bf A}_{11} & {\bf A}_{12}&  {\bf M}_{1}\\
{\bf A}_{21} & {\bf A}_{22} & {\bf M}_{2} \\[1mm]
{\bf M}^{T}_{1} & {\bf M}_{2}^{T} & 0
\end{array} 
\right )},
\end{equation}
where ${\bf M}_{i} = \big(a_{\omega}(\chi\varphi_{k}, \chi)\big)_{k\in \mathcal{B}_{i}}\in \mathbb{R}^{n_{i}} \;(i=1,2)$. The new matrix $\widehat{\bf A}$ is symmetric and invertible and thus we can perform an $LDL^{T}$ decomposition $\widehat{\bf A}={\bf L}{\bf D}{\bf L}^{T}$. As a result, the linear system \cref{eq:5-12} can be solved as follows:
\begin{equation}\label{eq:5-13}
{\left( \begin{array}{c}
{\bm \phi}_{1} \\
{\bm \phi}_{2} \\
c
\end{array} 
\right )} = {\bf L}^{-T}({\bf L}{\bf D})^{-1}
{\left( \begin{array}{c}
{\bf 0} \\
-\mathbf{A}_{21}{\bm p}\\
0
\end{array} 
\right )}.
\end{equation}
Finally, we need to add the constant function corresponding to the eigenvalue 0 to the local approximation space as it has been excluded from the eigenfunctions due to the constraint $a_{\omega}(\chi \phi_{h}, \chi) = 0$.

\begin{rem}
The size of the reduced eigenproblem \cref{eq:5-4}\,-\,\cref{eq:5-5} is about half of that of the augmented eigenproblem \cref{eq:5-3} and the matrix ${\bf A}_{11}$ in the reduced eigenproblem is positive definite, thereby making it much easier to solve. In practice, a sparse $LDL^{T}$ factorization \cite{davis2016survey} with a permutation of the rows and columns of the sparse matrix $\mathbf{A}$ can be used, which can greatly reduce the number of nonzeros in the Cholesky factor and the work in the factorization. Furthermore, we can always subdivide the subdomains such that the resulting matrices are amenable to $LDL^{T}$ factorization.
\end{rem}

In contrast to previous numerical attempts \cite{babuska2011optimal,babuvska2020multiscale,calo2016randomized,chen2020randomized} to solve the eigenproblem by first approximating the discrete $A$-harmonic spaces, we incorporate the discrete $A$-harmonic constraint into the eigenproblem here. Furthermore, instead of solving the augmented eigenproblem directly, we solve the equivalent reduced eigenproblem at a much lower computational cost. Therefore, our method is more efficient, and more importantly, more accurate without introducing additional errors. Building on the accurate solution of the eigenproblem, we see from \cref{thm:3-2} that the error in approximating the fine-scale solution locally by the local approximation space is bounded by $1/\lambda^{1/2}_{h,n_{i}+1}$, where $\lambda_{h,n_{i}+1}$ denotes the eigenvalue corresponding to the first eigenfunction that is not included in the local approximation space. Thus, we are able to obtain $a \;posteriori$ an upper bound on the local approximation error from the computed eigenvalues, which can then be used as a criterion for determining how many eigenfunctions to include. This is a significant advantage of the MS-GFEM over other multiscale methods, provided the eigenproblems are solved accurately.

Another noteworthy feature of our method is that due to the nearly exponential decay of the eigenvalues (see \cref{thm:3-2,thm:4-2-1}), the relative gap between two consecutive
eigenvalues is in general large. This is a favorable property that is not enjoyed by general eigenvalue problems. It ensures that classical iterative methods for solving the eigenproblems, such as subspace iteration or the Lanczos method, converge very fast.

\section{Numerical experiments}\label{sec-6}
\begin{figure}[!htbp]
\centering
\includegraphics[scale=0.35]{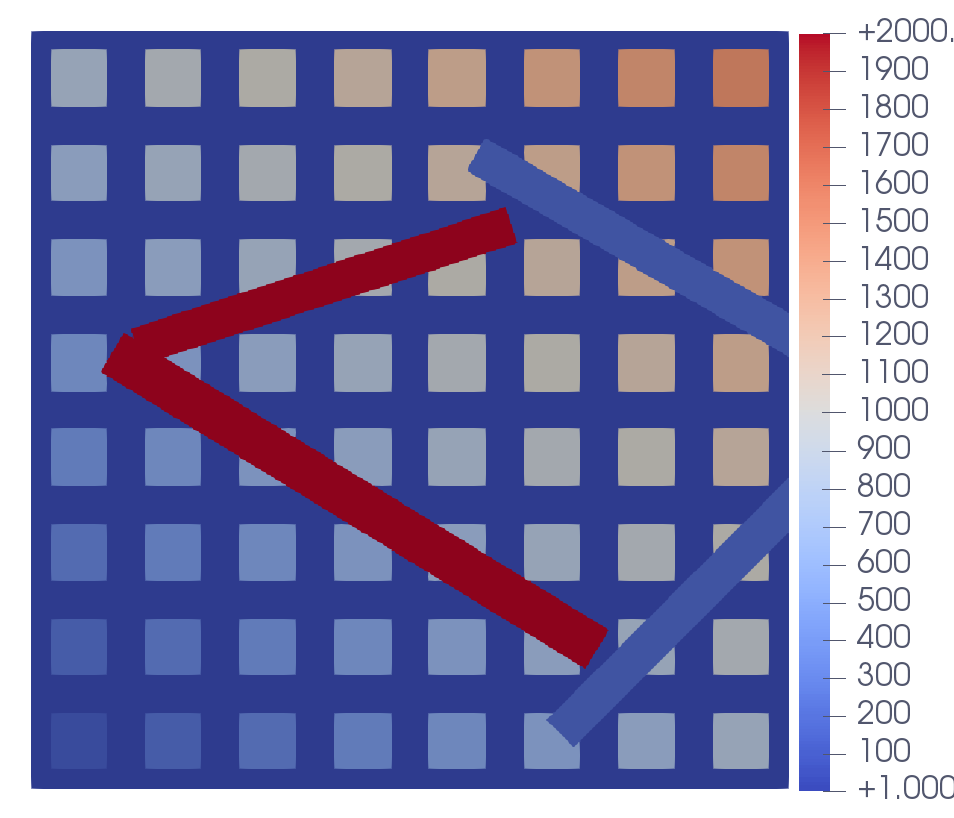}~\hspace{2ex}
\includegraphics[scale=0.35]{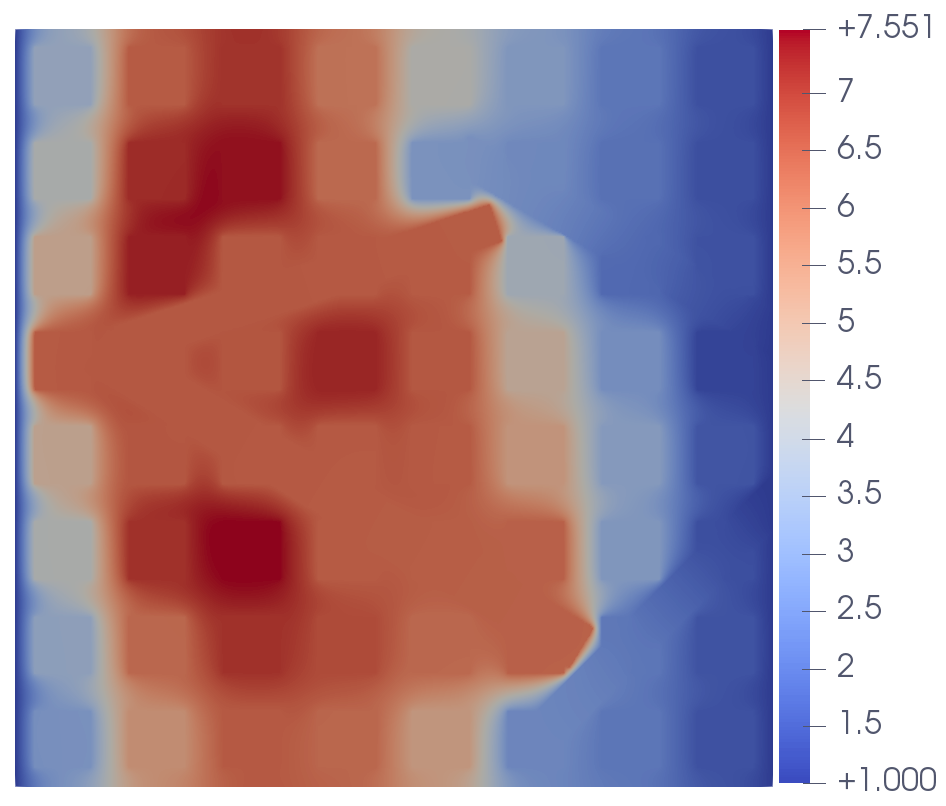}
\caption{The coefficient $A({\bm x}) = a({\bm x}) I$ (left) and the fine-scale solution $u^{e}_{h}$ (right).}\label{fig:6-1}
\end{figure}
In this section, we present numerical examples to confirm the theoretical results and to demonstrate the effectiveness of the method. We consider the following steady-state heat conduction problem on $\Omega=[0,1]^{2}$:
\begin{equation}\label{eq:6-1}
\left\{
\begin{array}{lll}
{\displaystyle -{\rm div}(A({\bm x})\nabla u({\bm x})) = f({\bm x}),\;\quad {\rm in}\,\;\Omega }\\[2mm]
{\displaystyle {\bm n} \cdot A({\bm x})\nabla u({\bm x})=1, \qquad \quad \quad \;{\rm on}\;\,\partial \Omega_{N}}\\[2mm]
{\displaystyle u({\bm x}) = 1, \;\;\;\qquad \qquad \qquad \qquad {\rm on}\;\,\partial \Omega_{D},}
\end{array}
\right.
\end{equation}
where $\partial \Omega_{N} = \big\{(x_{1}, x_{2})\in \Omega: x_{2}=0 \;{\rm or}\; 1\big\}$ and $\partial \Omega_{D} = \big\{(x_{1}, x_{2})\in \Omega: x_{1}=0 \;{\rm or}\; 1\big\}$. The high-contrast heterogeneous coefficient $A({\bm x})$ is displayed in \cref{fig:6-1} (left) and the source term $f({\bm x})$ is given by
\begin{equation}\label{eq:6-2}
f({\bm x}) = 10^{3}\times\exp\big(-10(x_{1}-0.15)^{2}-10(x_{2}-0.55)^{2}\big),
\end{equation}
leading to the multiscale fine-scale solution $u^{e}_{h}$ displayed in \cref{fig:6-1} (right).

All the local computations in the discrete MS-GFEM are carried out on a fine uniform Cartesian grid $\tau_{h}$ with $h=1/500$. The domain is first divided into $M=m^{2}$ square non-overlapping domains resolved by the mesh, which are then overlapped by 2 layers of mesh elements to form an overlapping decomposition $\{ \omega_{i}\}$. The overlapping subdomains $\omega_{i}$ are extended by $\ell$ layers of fine mesh elements to create the larger oversampling domains $\omega_{i}^{\ast}$ on which the local problems are solved. The local approximation space on each subdomain is constructed by $n_{\rm loc}$ eigenfunctions of the eigenproblem \cref{eq:3-5} where a local partition of unity operator is used \cite{ma2021novel}. We consider the fine-scale FE approximation $u^{e}_{h}$ as the reference solution and define the error between $u^{e}_{h}$ and the discrete MS-GFEM approximation $u^{G}_{h}$ as 
\begin{equation}\label{eq:6-3}
\mathbf{error} := {\Vert u^{e}_{h} - u_{h}^{G}\Vert_{H^{1}(\Omega)}}\,\big/\,{\Vert u^{e}_{h}\Vert_{H^{1}(\Omega)}}.
\end{equation}

\begin{figure}
\begin{center}
\includegraphics[scale = 0.32]{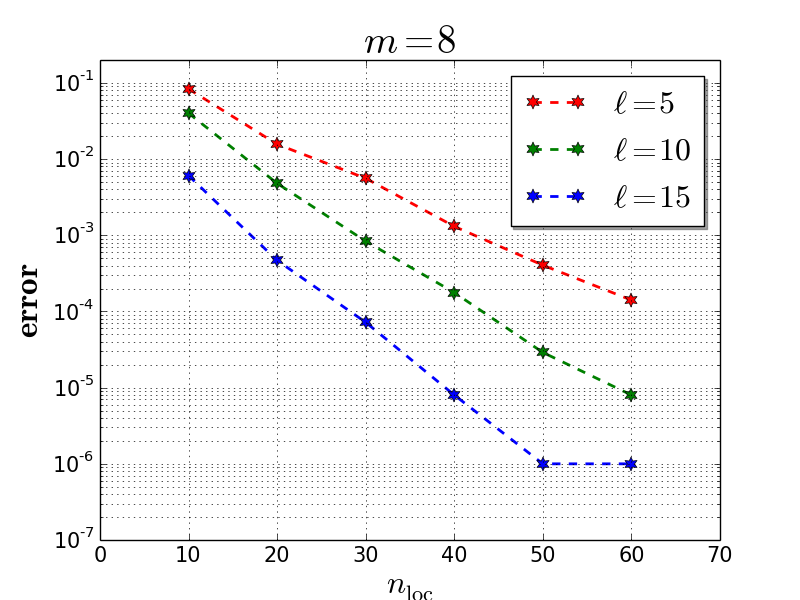}~\hspace{-5mm}
\includegraphics[scale=0.32] {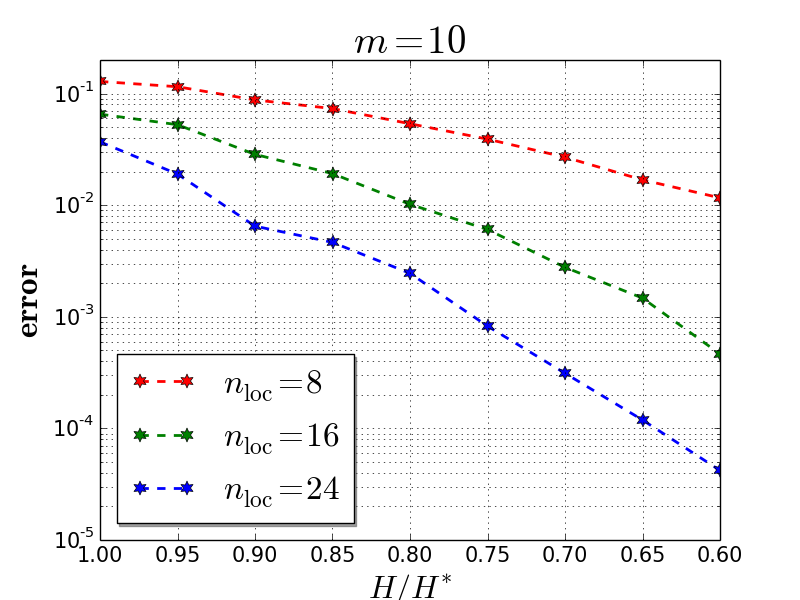}
\caption{Left: $\mathbf{error}$ versus $n_{\rm loc}$ with $m=8$; right: $\mathbf{error}$ versus $H/H^{\ast}$ with $m=10$.}\label{fig:6-2}
\end{center}
\end{figure}

In \cref{fig:6-2} (left), the errors are plotted as functions of the dimension of local spaces on a semilogarithmic scale for different oversampling sizes and a fixed number of subdomains. We observe that the errors decay with respect to $n_{\rm loc}$ at a rate of $\exp(-bn_{\rm loc})$, which is much faster than the rate guaranteed by \cref{thm:4-2-1}. Next we vary the oversampling size and plot the errors as functions of $H/H^{\ast}$ for different dimensions of local spaces with $m=10$ in \cref{fig:6-2} (right). Here $H$ and $H^{\ast}$ denote the side lengths of the subdomains $\omega_{i}$ and the oversampling domains $\omega_{i}^{\ast}$, respectively. We observe that the errors decay nearly exponentially with respect to $H/H^{\ast}$ as expected and the decay rate is higher with larger $n_{\rm loc}$, which agrees well with the established theoretical analysis.

\begin{figure}
\begin{center}
\includegraphics[scale=0.32] {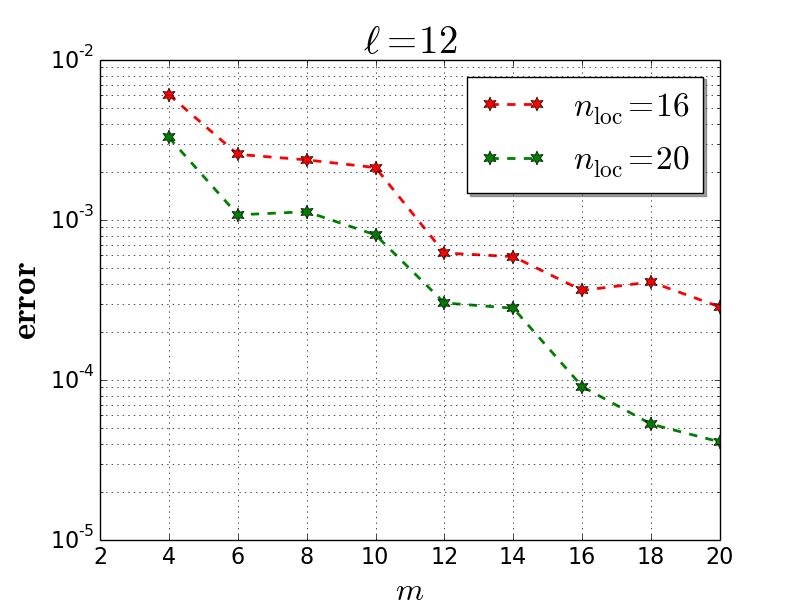}~\hspace{-5mm}
\includegraphics[scale=0.32] {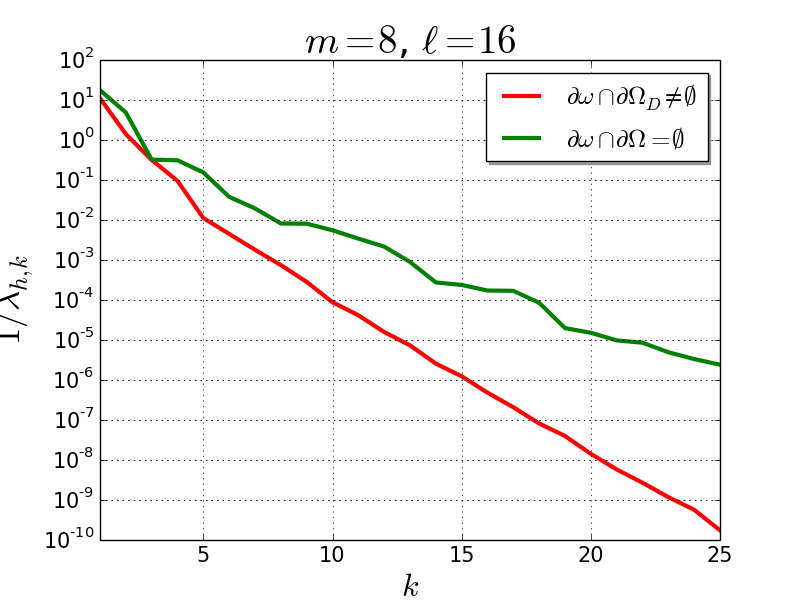}
\caption{Left: $\mathbf{error}$ versus $m$ with $\ell=12$; right: $1/\lambda_{h,k}$ versus $k$ with $m=8$ and $\ell=16$. $M=m^{2}$ is the number of subdomains and $\lambda_{h,k}$ denotes the $k$-th eigenvalue of the local eigenproblems.}\label{fig:6-3}
\end{center}
\end{figure}

In \cref{fig:6-3} (left), we study how the error varies with the number of subdomains ($M=m^{2}$) for a fixed oversampling size. One can observe that overall, the errors decrease as the number of subdomains increases. Note that in this case, the quantity $H/H^{\ast}$ decreases as $m$ increases since the oversampling size $\ell$ is fixed. Finally, we show the reciprocals of the eigenvalues of the local eigenproblems in an interior subdomain and a subdomain intersecting the Dirichlet boundary $\partial \Omega_{D}$ in \cref{fig:6-3} (right), again on a semilogarithmic scale. We clearly see that the eigenvalues decay rapidly and that the nearly exponential decay rate is higher for the subdomain intersecting $\partial \Omega_{D}$ than for the interior subdomain. 


\section{Conclusions}
In this paper, error estimates of the discrete MS-GFEM based on FE approximations of the local eigenvalue problems have been derived. It has been shown that the error of the discrete MS-GFEM is bounded by the error of the fine-scale FE approximation of the exact solution and the local errors in approximating the fine-scale solution. The local approximation errors in the discrete method converge toward those in the continuous method as $h\rightarrow 0$ and decay nearly exponentially with respect to the dimension of the local spaces. To solve the local eigenvalue problems efficiently and accurately, a novel method based on a mixed formulation of the generalized eigenproblems has been proposed. 

In a similar way that the continuous MS-GFEM aims to approximate the exact continuous solution, the discrete MS-GFEM aims to approximate the fine-scale FE solution. The significance of the nearly exponential decay rate of the local approximation errors is that typically only a handful of eigenfunctions are required per subdomain to achieve moderate error tolerances, leading to a small overall dimension of the discrete MS-GFEM space. In fact, the discrete MS-GFEM can be considered as an efficient two-level domain-decomposition type method for approximating the finite element solution of the problem. Compared to the traditional two-level overlapping Schwarz methods, our method solves the local problems and the global coarse problem only once without using Krylov methods. Therefore, the computation and communication costs in a parallel implementation can be dramatically reduced. Moreover, the local approximation spaces can be reused if many problems with different right-hand sides need to be solved, or if the coefficient changes only locally in some parts of the domain.

\appendix

\bibliographystyle{siamplain}
\bibliography{references}
\end{document}